\documentclass[11pt,reqno]{amsart}

\usepackage[T1]{fontenc}

\usepackage{mathtools}
\usepackage{amsmath}								
\usepackage{amssymb}
\usepackage{amsthm}
\usepackage{amscd}
\usepackage{amsfonts}
\usepackage{stmaryrd}
\usepackage{algorithm, algorithmic}
\usepackage{ wasysym }
\usepackage{listings}

\usepackage{caption}
\usepackage{subcaption}
\usepackage{euler}

\usepackage{extarrows}

\usepackage[colorlinks, linktocpage, citecolor = red, linkcolor = blue]{hyperref}
\usepackage{color}

\usepackage{tikz}									
\usetikzlibrary{matrix}
\usetikzlibrary{patterns}
\usetikzlibrary{matrix}
\usetikzlibrary{positioning}
\usetikzlibrary{decorations.pathmorphing}
\usetikzlibrary{cd}

\usepackage{fullpage}
\usepackage[shortlabels]{enumitem}

\newtheorem{maintheorem}{Theorem}									
\newtheorem{theorem}{Theorem}[section]
\newtheorem{lemma}[theorem]{Lemma}
\newtheorem{corollary}[theorem]{Corollary}

\newtheorem{proposition}[theorem]{Proposition}

\theoremstyle{definition}
								
\newtheorem{definition}[theorem]{Definition}

\newtheorem{example}[theorem]{Example}

\newtheorem{remark}[theorem]{Remark}

\newcommand{\RR}{\mathbb{R}}

\title{Voronoi Cells in Metric Algebraic Geometry of Plane Curves}

\author{Madeline Brandt}
\address{Department of Mathematics, Brown University}
\email{\href{mailto:madeline\_brandt@brown.edu}{madeline\_brandt@brown.edu}}

\author{Madeleine Weinstein}
\address{Department of Mathematics, Stanford University}
\email{\href{mailto:mweinste@stanford.edu}{mweinste@stanford.edu}}

\begin{document}

\begin{abstract} 
Voronoi cells of varieties encode many features of their metric geometry. 
We prove that each Voronoi or Delaunay cell of a plane curve appears as the limit of a sequence of cells obtained from point samples of the curve.
We use this result to study metric features of plane curves, including the medial axis, curvature, evolute, bottlenecks, and reach. In each case, we provide algebraic equations defining the object and, where possible, give formulas for the degrees of these algebraic varieties. We show how to identify the desired metric feature from Voronoi or Delaunay cells, and therefore how to approximate it by a finite point sample from the variety.
\end{abstract}

\maketitle

\setcounter{tocdepth}{1}

\section{Introduction}

\emph{Metric algebraic geometry} addresses questions about real algebraic varieties involving distances. For example, given a point $x$ on a real algebraic plane curve $X \subset \mathbb{R}^2$, we may ask for the locus of points which are closer to $x$ than to any other point of $X$. This is called the \emph{Voronoi cell of $X$ at $x$} \cite{voronoi}.
The boundary of a Voronoi cell consists of points which have more than one nearest point to $X$. So we may ask, given a point in $\mathbb{R}^2$, how close must it be to $X$ in order to have a unique nearest point on $X$? This quantity is called the \emph{reach}, and was first defined in \cite{Federer}. 

We use Voronoi cells and their duals, Delaunay cells (see Definition \ref{def:delaunay}), to study metric features of plane curves. The following theorem makes precise the idea behind Figures \ref{fig:voronoi_butterfly} and \ref{fig:delaunay_butterfly}. 
\begin{maintheorem}
\label{thm:convergence}
Let $X$ be a compact algebraic curve in $\mathbb{R}^2$ and $\{A_\epsilon\}_{\epsilon \searrow 0}$ be a sequence of finite subsets of $X$ containing all singular points of $X$ such that every point of $X$ is within distance $\epsilon$ of some point in $A_\epsilon$.
\begin{enumerate}
    \item Every Voronoi cell is the Wijsman limit (see Definition \ref{def:wijsman}) of a sequence of Voronoi cells of $\{A_\epsilon\}_{\epsilon \searrow 0}$.
    \item If $X$ is not tangent to any circle in four or more points, then every maximal Delaunay cell is the Hausdorff limit (see Definition \ref{def:hausdorff}) of a sequence of Delaunay cells of $\{A_\epsilon\}_{\epsilon \searrow 0}$.
\end{enumerate}
\end{maintheorem}

\begin{figure}[h]
\centering
\begin{subfigure}{.33\textwidth}
  \centering
  \includegraphics[width=\linewidth]{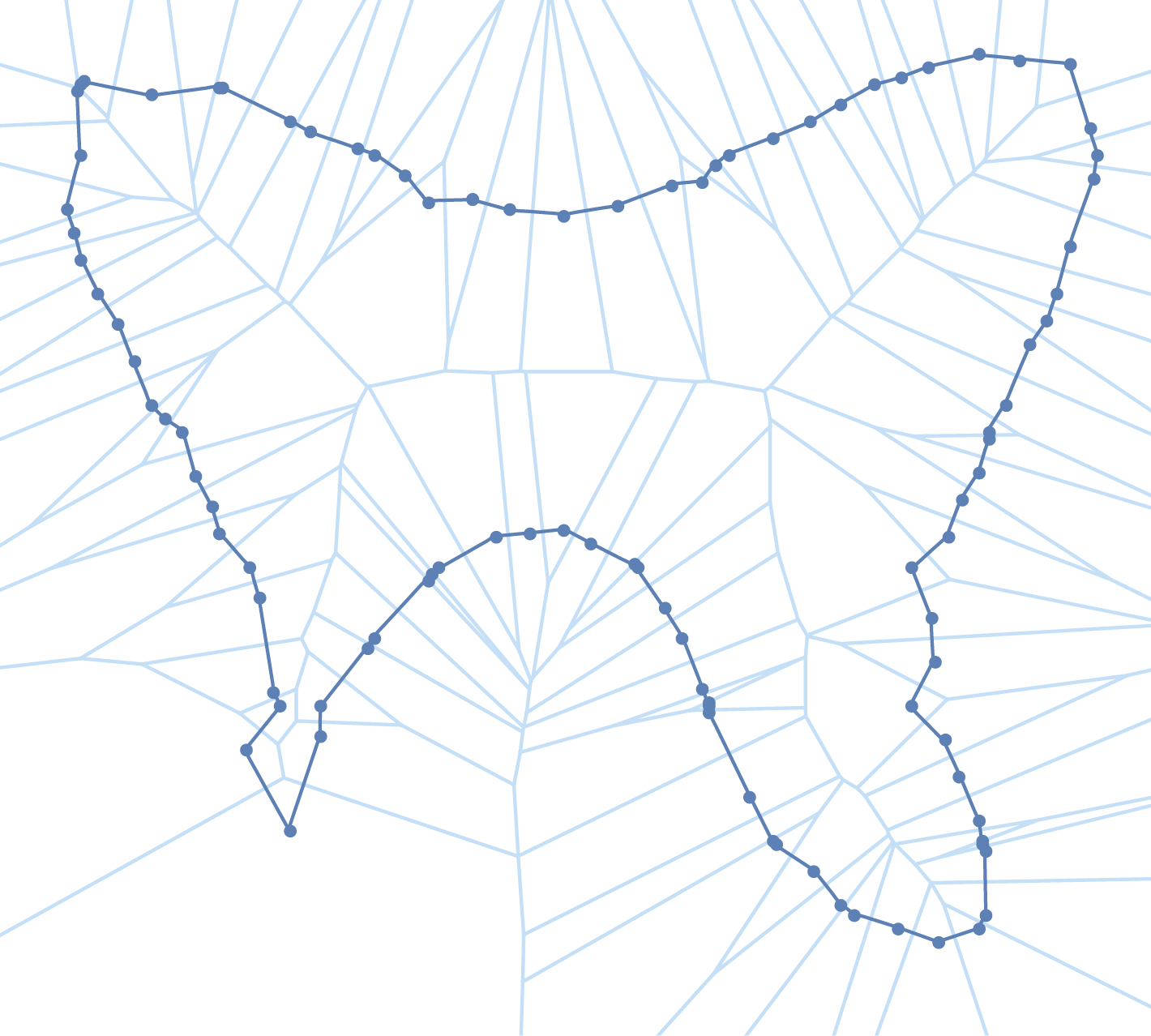}
\end{subfigure}%
\begin{subfigure}{.33\textwidth}
  \centering
  \includegraphics[width=\linewidth]{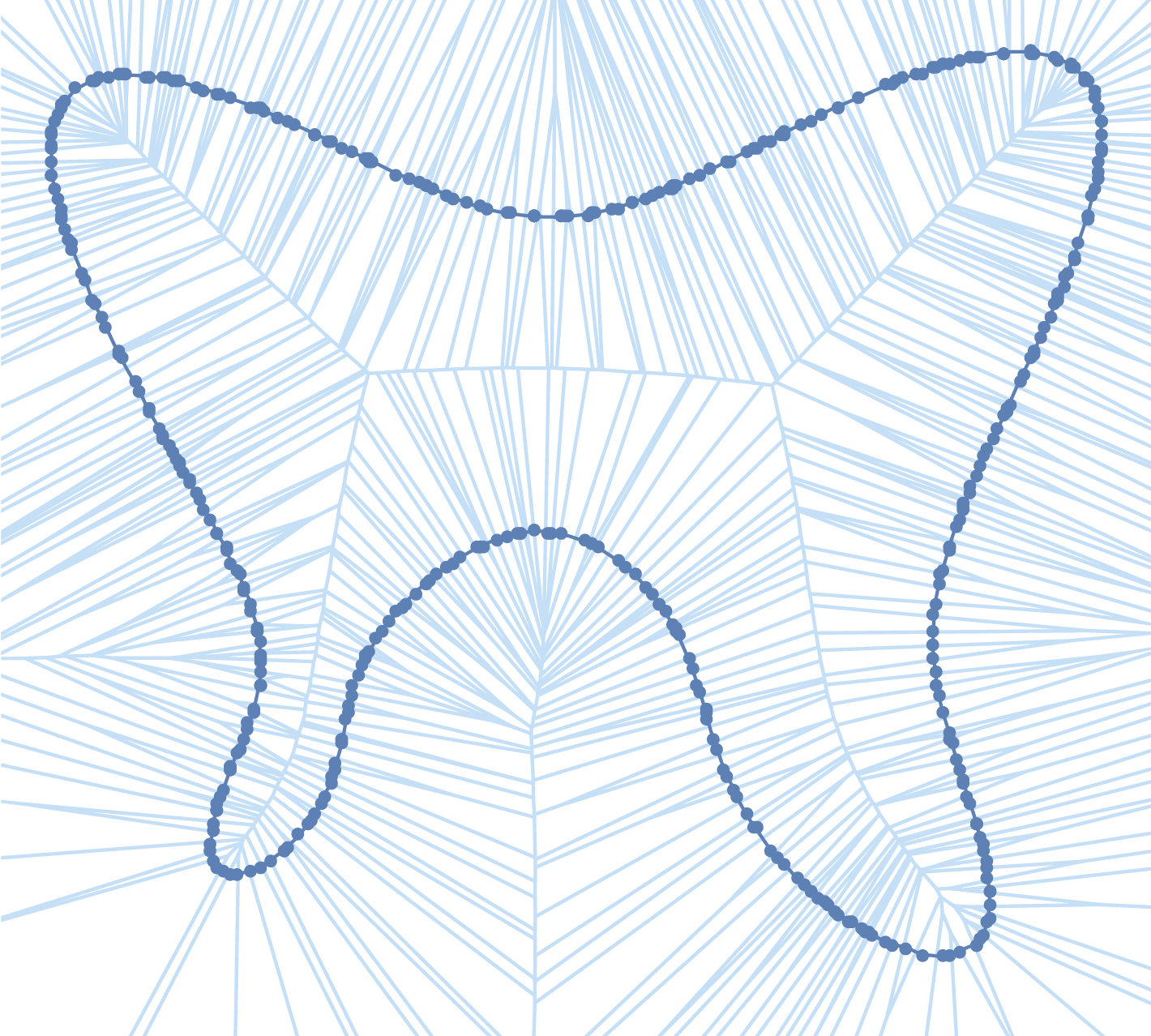}
\end{subfigure}
\begin{subfigure}{.33\textwidth}
  \centering
  \includegraphics[width=\linewidth]{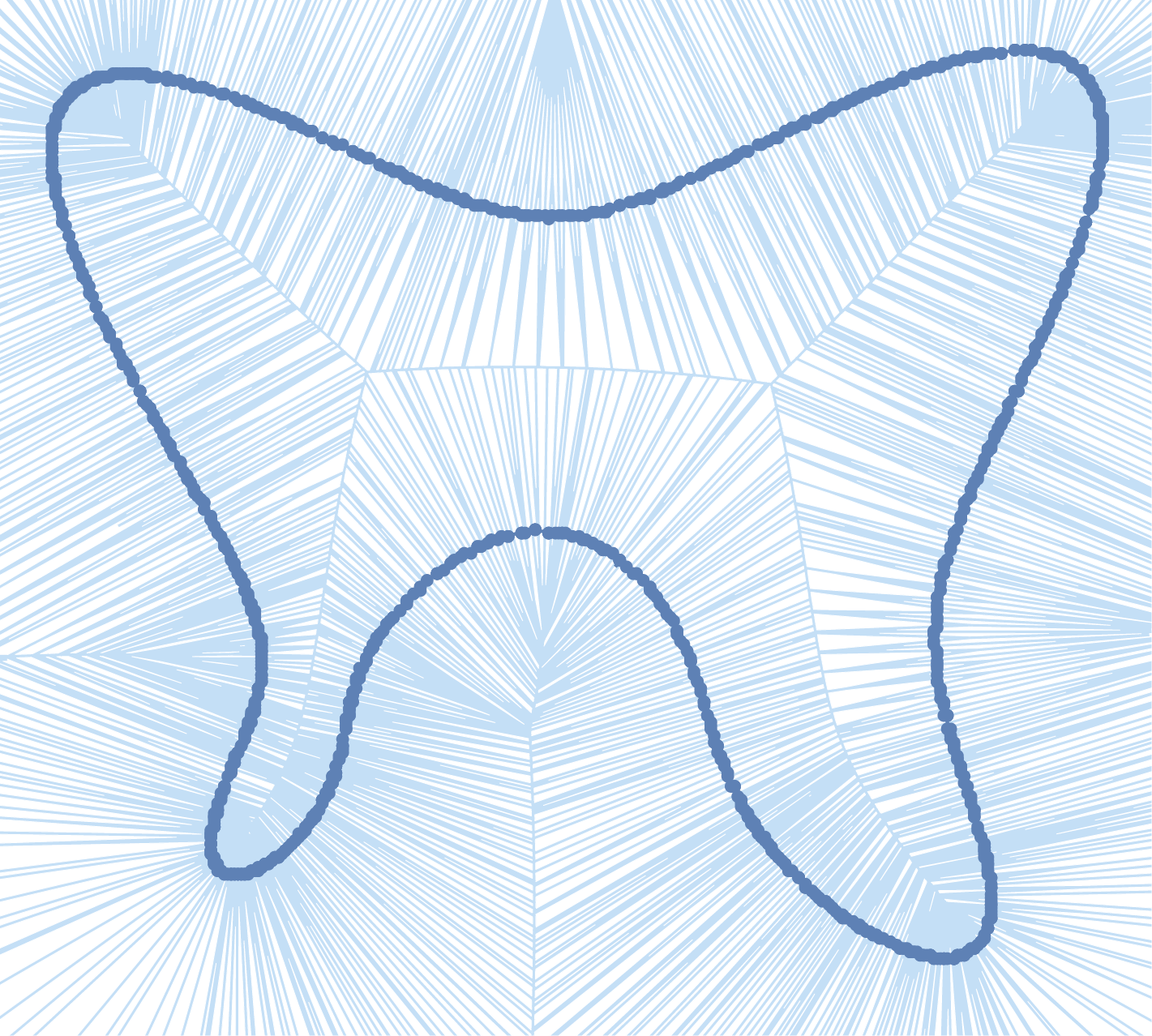}
\end{subfigure}
\caption{Voronoi cells of 101, 441, and 1179 points sampled from the butterfly curve.}
\label{fig:voronoi_butterfly}
\end{figure}

\begin{figure}[h]
\centering
\begin{subfigure}{.33\textwidth}
  \centering
  \includegraphics[width=\linewidth]{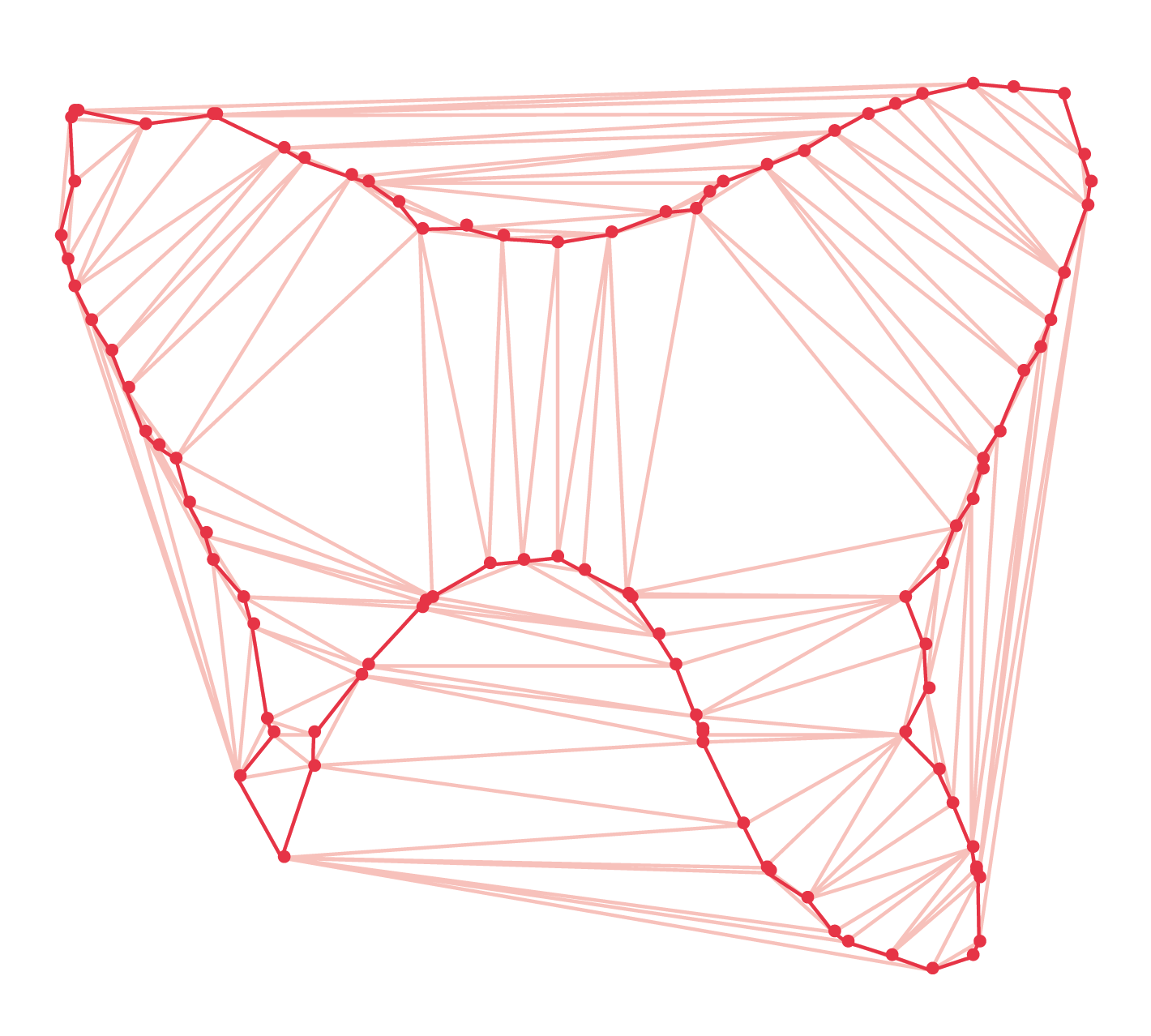}
\end{subfigure}%
\begin{subfigure}{.33\textwidth}
  \centering
  \includegraphics[width=\linewidth]{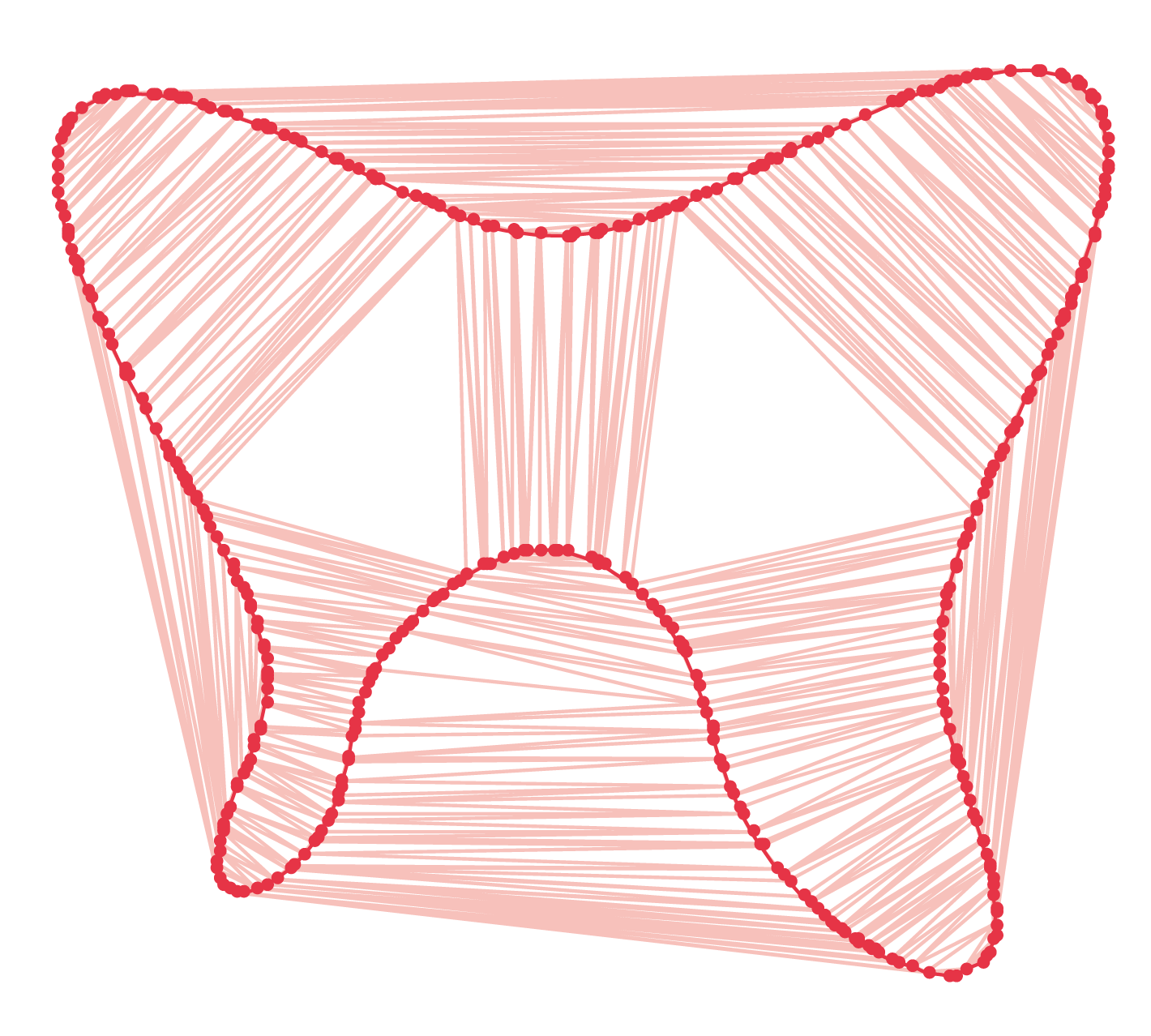}
\end{subfigure}
\begin{subfigure}{.33\textwidth}
  \centering
  \includegraphics[width=\linewidth]{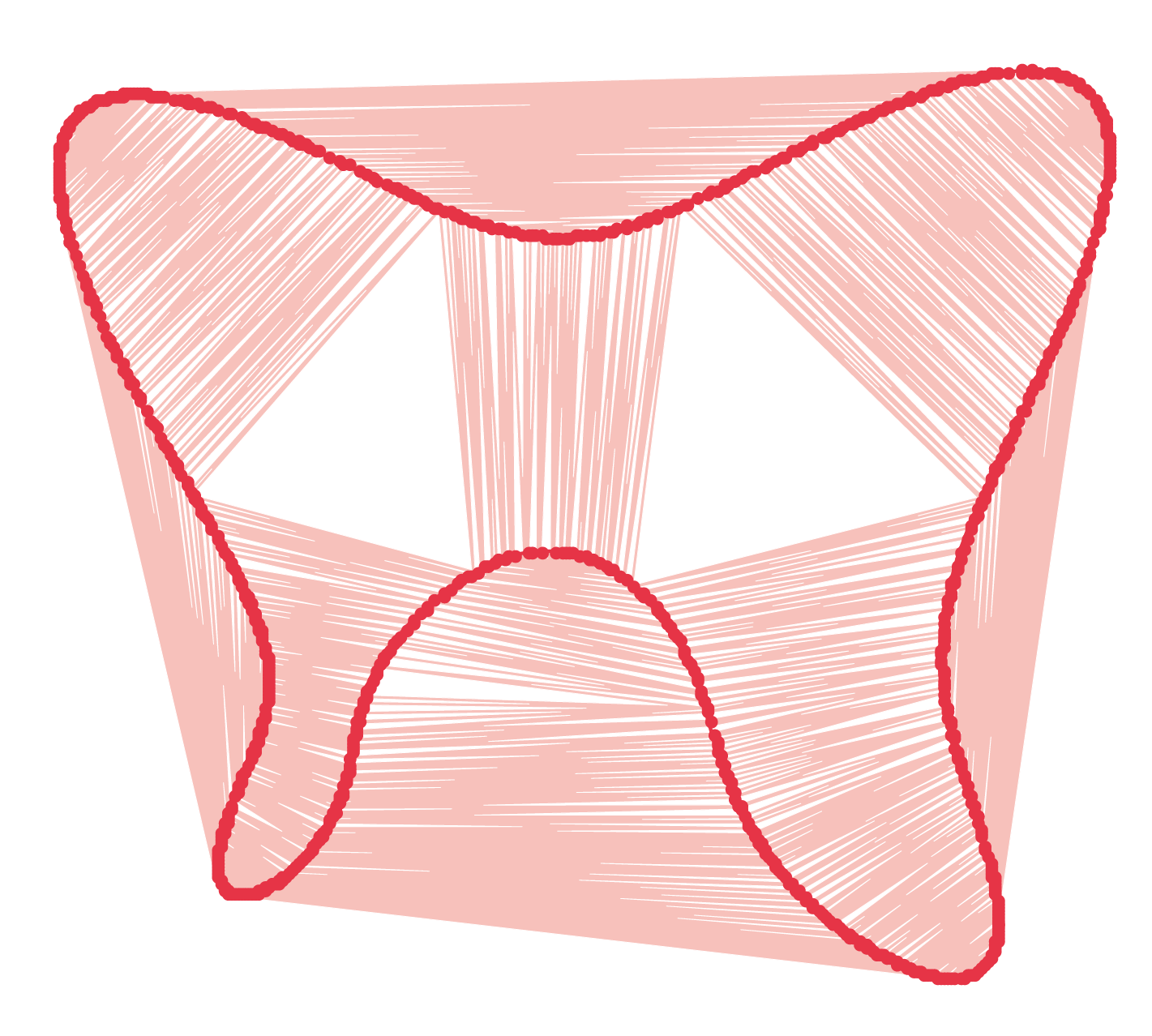}
\end{subfigure}
\caption{Delaunay cells of  101, 441, and  1179 points sampled from the butterfly curve. The two large triangles correspond to tritangent circles of the curve.}
\label{fig:delaunay_butterfly}
\end{figure}

Voronoi diagrams of finite point sets are widely studied and have seen applications across science and technology, most notably in the natural sciences, health, engineering, informatics, civics and city planning. For example in Victoria, a state in Australia, students are typically assigned to the school to which they live closest. Thus, the catchment zones for schools are given by a Voronoi diagram \cite{VSD}. 
Metric features of varieties, such as the medial axis and curvature of a point, can be detected from the Voronoi cells of points sampled densely from a variety.  Computational geometers frequently use Voronoi diagrams to approximate these features and reconstruct varieties \cite{skeletons, skeleton_approx, normal_approximation}.

The reach of an algebraic variety is an invariant that is important in applications of algebraic topology to data science. For example, the reach determines the number of sample points needed for the technique of persistent homology to accurately determine the homology of a variety \cite{NSW}. For an algebraic geometric perspective on the reach, see \cite{BKSW}. The \emph{medial axis} of a variety is the locus of points which have more than one nearest point on $X$. This gives the following definition of the reach. 

\begin{figure}[h!]
    \centering
    \includegraphics[height = 2 in]{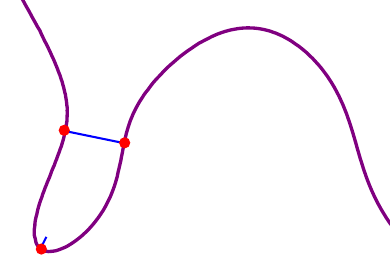}
    \caption{The reach of the butterfly curve is attained by the maximal curvature point in the lower left wing. The narrowest bottleneck is also shown. This figure is explained in Example \ref{ex:reach}.}
    \label{fig:reach}
\end{figure}

\begin{definition}
The \emph{reach} $\tau(X)$ of an algebraic variety $X \subset \mathbb{R}^n$ is the infimum of the set of distances from any point on $X$ to a point on the medial axis of $X$.
\end{definition}

The paper \cite{ACKMRW17} describes how the reach is the minimum of two quantities. 
We have 
\begin{equation}
\label{eqn:reach_min}
   \tau(X) = \min\left\{ q , \frac{\rho}{2}\right\}, 
\end{equation}
where $q$ is the minimum radius of curvature (Definition \ref{def:curvature}) of points in $X$ and $\rho$ is the narrowest bottleneck distance (Definition \ref{def:bottleneck_distance}). An example is depicted in Figure \ref{fig:reach}.

The paper is organized as follows. We begin with a systematic treatment in Section \ref{sec:limits} of convergence of Voronoi cells of increasingly dense point samples of a variety. Here, we also introduce Delaunay cells, which are dual to Voronoi cells but are compact, and thus able to exhibit Hausdorff convergence. This gives the proof of Theorem \ref{thm:convergence}, split among Theorems \ref{thm:delaunay_convergence} and \ref{thm:voronoi_convergence}, as well as Proposition \ref{prop:singular_convergence}, which treats the singular case separately. Theorem \ref{thm:convergence} is robust because it is not affected by the distribution of the point sample.
Theorem \ref{thm:convergence} provides the theoretical foundations for estimating metric features of a variety from a point sample. We do this for the medial axis (Section~\ref{sec:medial_axis}), curvature and evolute (Section~\ref{sec:evolute}), bottlenecks (Section~\ref{sec:bottlenecks}) and reach (Section~\ref{sec:reach}). For each of these metric features, we first give defining equations and where possible a formula for the degree. We then turn our attention to detecting information about a real plane curve $X$ from its Voronoi cells. For each metric feature, we state a theoretical result about how to detect the feature from the Voronoi cells of $X$ or a subset of $X$. Corollaries to Theorem~\ref{thm:convergence} provide convergence results for these features. The overall aim is to give a path to compute the metric features of an algebraic plane curve $X$ from Voronoi cells of dense point samples of $X$. We use the \emph{butterfly curve}
\begin{equation}
\label{eqn:butterfly}
   b(x,y) = x^4 - x^2 y^2 + y^4 - 4 x^2 - 2 y^2 - x - 4 y + 1 
\end{equation}
in our examples. 
 In computational geometry and data science, these problems are often considered when there is noise in the sample. In this paper we assume that our samples lie precisely on the curve $X$.


\section{Voronoi and Delaunay Cells of Varieties and Their Limits}\label{sec:limits}

Let $X \subset \mathbb{R}^n$ be a nonempty \textit{real algebraic variety,} that is, the zero locus of a set of polynomial equations with real coefficients. We note that this definition allows a variety to be reducible. We call $X$ a \textit{curve} if it is of dimension $1$ and say $X$ is \textit{smooth} if the real locus is smooth.  Let $d(x,y)$ denote the Euclidean distance between two points $x,y \in \mathbb{R}^n$.
\begin{definition}
The \emph{Voronoi cell of $x \in X$} is
$$
Vor_X(x) = \{y \in \mathbb{R}^n\ |\ d(y,x) \leq d(y,x') \text{ for all }x' \in X\}.
$$
\end{definition}
An example of a Voronoi cell is given in Figure \ref{fig:ellipse_ex}. 
This is a convex semialgebraic set whose dimension is equal to $codim(X)$ so long as $x$ is a smooth point of $X$. It is contained in the \emph{normal space to $X$ at $x$}:
$$
N_X(x) = \{u \in \mathbb{R}^n \ |\ u-x\text{ is perpendicular to the tangent space of }X\text{ at }x\}. 
$$
The topological boundary of the Voronoi cell $Vor_X(x)$ consists of the points in $\mathbb{R}^n$ that have two or more closest points in $X$, one of which is $x$. The collection of boundaries of Voronoi cells is described as follows.
\begin{definition}
\label{def:medial}
The \emph{medial axis} $M(X)$ of an algebraic variety $X \subset \mathbb{R}^n$ is the collection of points in $\mathbb{R}^n$ that have two or more closest points in $X$. An example of the medial axis is given in Figure \ref{fig:ellipse_ex}. 
\end{definition}
 
Let $B(p,r)$ denote the open disc with center $p \in \mathbb{R}^n$ and radius $r > 0$. We say this disc is \emph{inscribed} with respect to $X$ if $X \cap B(p,r) = \emptyset$ and we say it is \emph{maximally inscribed} if no disc containing $B(p,r)$ shares this property.
Each inscribed disc gives a Delaunay cell, defined as follows.

\begin{definition}\label{def:delaunay}
Given an inscribed disc $B$ of an algebraic variety $X \subset \mathbb{R}^n$, 
the \emph{Delaunay cell $\text{Del}_X(B)$} 
is
$
conv(\overline{B} \cap X).
$ 
An example of a Delaunay cell and the corresponding maximally inscribed disc is given in Figure \ref{fig:ellipse_ex}. 
\end{definition}

\begin{figure}[h]
    \centering
    \includegraphics[height=1.5 in]{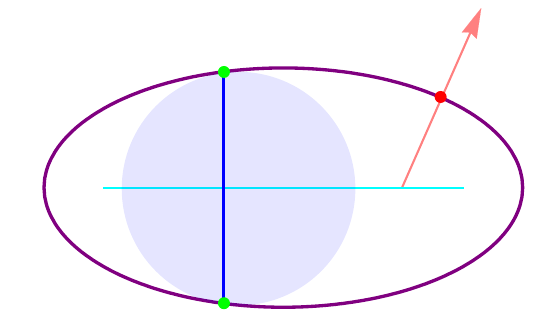}
    \caption{The ellipse $
(x/2)^2 + y^2 -1 = 0
$
is shown in purple. The Voronoi cell of the red point $(\sqrt{7}/2,3/4)$ is shown in pink. It is a ray starting at the point $(3\sqrt{7}/8,0)$ in the direction $(\sqrt{7}/4,3/2)$. The dark blue line segment between the points $(-1/2, \sqrt{15}/4)$ and $(-1/2, -\sqrt{15}/4)$ is a Delaunay cell defined by the light blue maximally inscribed circle with center $(-3/8,0)$ and radius $\sqrt{61}/8$. The light blue line is the medial axis, which goes from $(-3/2,0)$ to $(3/2,0)$ because the curvature at the points $(-2,0)$ and $(2,0)$ is 2.
}
    \label{fig:ellipse_ex}
\end{figure}

\begin{remark}
For plane curves, the collection of centers of all inscribed spheres which give maximal Delaunay cells (Delaunay cells which are not contained in any other Delaunay cell) is the Euclidean closure of the medial axis. Points of an algebraic plane curve $X$ which are themselves maximal Delaunay cells are points of $X$ with locally maximal curvature. In this case, the maximally inscribed circle is an \emph{osculating circle}, see Definition \ref{def:curvature}.
\end{remark}

We now describe two convex sets whose face structures encode the Delaunay and Voronoi cells of $X$.
We embed $\mathbb{R}^n$ in $\mathbb{R}^{n+1}$ by adding a coordinate. We usually imagine that this last coordinate points vertically upwards.
So, we say that $x \in \mathbb{R}^{n+1}$ is below $y \in \mathbb{R}^{n+1}$ if $x_{n+1} \leq y_{n+1}$ and all other coordinates are the same. 
 Let 
$$
U = \{ x \in \mathbb{R}^{n+1} \ |\ x_{n+1} = x_1^2 + \cdots + x_n^2 \}
$$
be the \emph{standard paraboloid} in $\mathbb{R}^{n+1}$. 
If $p \in \mathbb{R}^n,$ then let $p_U = (p,||p||^2)$ denote its lift to $U$.

Given a convex set $C \subset \mathbb{R}^{n+1}$, a convex subset $F \subset C$ is called a \emph{face} of $C$ if for every $x \in F$ and every $y,z \in C$ such that $x \in conv(y,z),$ we have that $y,z \in F$. We say that a face $F$ is \emph{exposed} if there exists an \emph{exposing hyperplane} $H$ such that $C$ is contained in one closed half space of the hyperplane and such that $F = C \cap H$. We call an exposed face $F$ a  \emph{lower exposed face} of $C$ if there is an exposing hyperplane lying below $C$.

\begin{definition}
\label{def:lifted_delaunay}
The \emph{Delaunay lift of an algebraic variety $X \subset \mathbb{R}^n$} is the convex set 
$$
P^*_X = conv(x_u\ |\ x \in X) + \{(0,\ldots,0,\lambda)\ :\ \lambda \in \mathbb{R}_{\geq 0}\} \subset \mathbb{R}^{n+1},
$$
where we recall that $x_u = (x , || x || ^2)$ and use $+$ to denote the Minkowski sum.
 The Delaunay lift of the butterfly curve is shown in Figure \ref{fig:lifted_delaunay}.
\end{definition}

\begin{figure}[h]
\centering
\begin{minipage}{.55\textwidth}
  \centering
  \includegraphics[width=.8\linewidth]{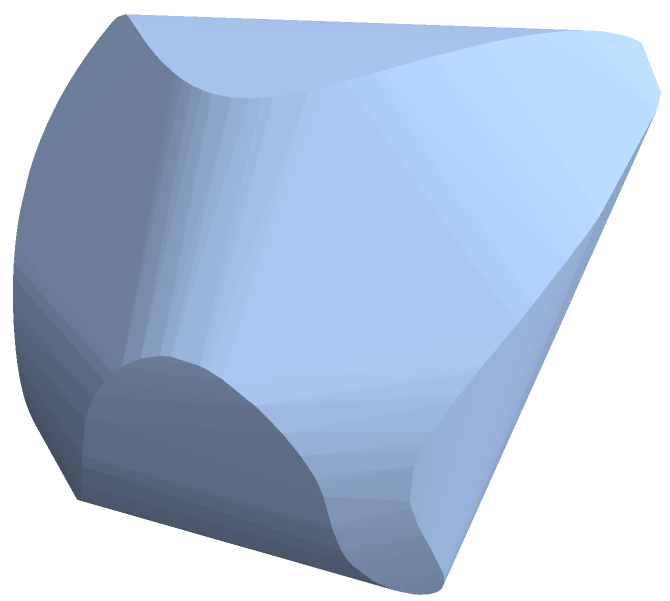}
  \captionof{figure}{The Delaunay lift (Definition \ref{def:lifted_delaunay}) of the butterfly curve, viewed from below.}
  \label{fig:lifted_delaunay}
\end{minipage}%
\begin{minipage}{.45\textwidth}
  \centering
  \includegraphics[width=.8\linewidth]{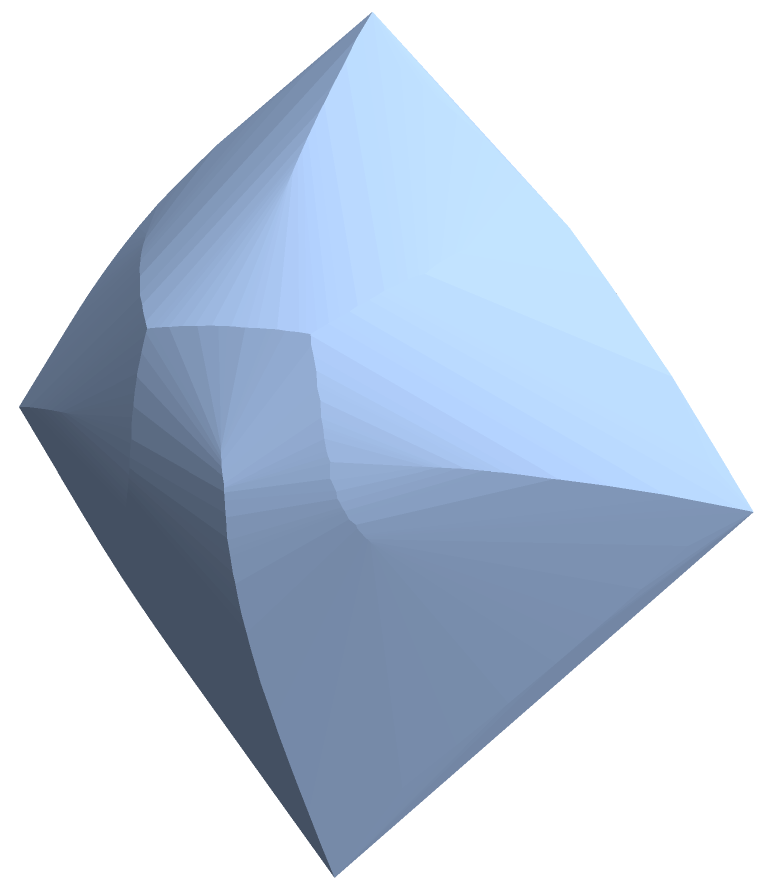}
  \captionof{figure}{The Voronoi lift (Definition \ref{def:lifted_voronoi}) of the butterfly curve, viewed from below.}
  \label{fig:lifted_voronoi}
\end{minipage}
\end{figure}

We now study how the lower exposed faces of the Delaunay lift $P^*_X$ project to $conv(X)$, and give the Delaunay cells of $X$.

\begin{proposition}
\label{prop:delaunay_faces}
Let $X \subset \mathbb{R}^n$ be an algebraic variety. Let $\pi: \mathbb{R}^{n+1} \rightarrow \mathbb{R}^n$ be the projection onto the first $n$ coordinates. A subset $F \subset P^*_X$ is a lower exposed face if and only if $\pi(F)$ is a Delaunay cell of $X$. Furthermore, if $H_F$ is the hyperplane exposing $F$, then $\pi(U\cap H_F)$ is an inscribed sphere of $X$ and $\pi(F) = \text{Del}_X(\pi(U\cap H_F)).$
\end{proposition}

\begin{proof}
The map from $\mathbb{R}^n \rightarrow \mathbb{R}^{n+1}$ defined by $x \mapsto x_U= (x,||x||^2)$ lifts every sphere in $\mathbb{R}^n$ to the intersection of a hyperplane $H$ with $U$ \cite[Proposition 7.17]{joswig}. Moreover, the projection of the intersection of any hyperplane $H$ (satisfying $\pi(H) = \mathbb{R}^n$) with 
$U$ gives a sphere in $\mathbb{R}^n$ \cite[Proposition 7.17]{joswig}.

Given a Delaunay cell $\text{Del}_X(B)$ for some inscribed sphere $B$, we have that $P^*_X$ lies above the corresponding hyperplane $H$. This is because any points below $H$ would project to points in $X$ lying inside of $B$, contradicting the condition that $X \cap B = \emptyset$ for an inscribed disc $B$. So, $H$ is the exposing hyperplane of the face $(Del_X(B))_U$.

Suppose $F \subset P^*_X$ is a lower exposed face with exposing hyperplane $H_F$. The interior of the sphere $\pi(U\cap H_F)$ contains no points of $X$, because if it did contain a point $x$, then $x_U$ would lie in the lower half-space of $H_F$, which does not intersect $P^*_X$. Then $\pi(U\cap H_F)$ is the inscribed disc corresponding to a Delaunay cell. 

Since $\pi(U\cap H_F)$ is a sphere, we have $\text{Del}_X(\pi(U\cap H_F)) = conv(\pi(U\cap H_F)\cap X)$. Let $X_U$ denote the lift of $X$ to $\mathbb{R}^{n+1}$. Then $\pi(U\cap H_F)\cap X = \pi(U\cap H_F\cap X_U) = \pi(H_F \cap X_U),$ and so $\text{Del}_X(\pi(U\cap H_F)) = conv(\pi(H_F \cap X_U)) = \pi(conv(H_F \cap X_U)) = \pi(F)$.
\end{proof}

We may define a convex set whose faces project down to the Voronoi cells as follows. For any point $x \in X$, let $T(x)$ denote the hyperplane in $\mathbb{R}^{n+1}$ through $x_U =  (x,||x||^2)$ tangent to the paraboloid $U$. Let $T(x)^+$ be the closed half-space consisting of all points in $\mathbb{R}^{n+1}$ lying above the hyperplane $T(x)$.
\begin{definition}
\label{def:lifted_voronoi}
The \emph{Voronoi lift of an algebraic variety $X \subset \mathbb{R}^n$} is the convex set
$P_X = \cap_{x \in X} T(x)^+$.
The Voronoi lift of the butterfly curve is shown in Figure \ref{fig:lifted_voronoi}. 
\end{definition}

The lower exposed faces of the Voronoi lift $P_X$ project to Voronoi cells of $X$, as we now show.

\begin{proposition}
\label{prop:voronoi_faces}
Let $X \subset \mathbb{R}^n$ be an algebraic variety. Let $\pi: \mathbb{R}^{n+1} \rightarrow \mathbb{R}^n$ be the projection onto the first $n$ coordinates. A subset $F$ of the Voronoi lift $P_X$ is an exposed face of $P_X$ if and only if $\pi(F)$ is a Voronoi cell of $X$. Furthermore, if $H_F$ is the hyperplane exposing $F$ and $U\cap H_F \neq \emptyset$,  then $U\cap H_F$ is a point and 
$
\pi(F) = \text{Vor}_X(\pi(U \cap H_F)).
$
\end{proposition}

\begin{proof}
For some point $x \in X$, consider $P_X \cap T(x)$. Let $p \in \mathbb{R}^n$. There exists $p' \in T(x)$ with $\pi(p_U) = \pi(p')$. The distance from  $p_U$ to the point $p'$ is the square of the distance $d_{\mathbb{R}^n}(\pi(p),\pi(x))$ \cite[Lemma 6.11]{joswig}.
Therefore, $P_X \cap T(x)$ consists of those points $p'$ for which the distance $d_{\mathbb{R}^n}(p,x)$ is minimal over all $x \in X$. In other words, ${\pi(P_X \cap T(x)) = Vor_X(x)}$. 

Suppose $F \subset P_X$ is an exposed face with exposing hyperplane $H_F$ such that $U\cap H_F \neq \emptyset$. Let $p \in U\cap H_F$. Since $U \subset P_X$ we have that $p \in P_X$. Then, $p \in F = H_F \cap P_X$. This implies $H_F$ is the tangent hyperplane to $U$ at the point $p$, so in particular, $p=U \cap H_F$. Since $p$ is on the boundary of $P_x$, we have $\pi(p)\in X$ and $T(\pi(p)) = H_F$. We have $\pi(F)= \pi(P_X \cap T(\pi(p))) =  Vor_X(\pi(p)) = Vor_X(\pi(U \cap H_F))$, where in the second equality we use the result in the preceding paragraph. 
\end{proof}

There is a sense in which the Voronoi lift $P_X$ and the Delaunay lift $P^*_X$ are dual. We now describe this relationship. Suppose that $X$  is not contained in any proper linear subspace of $\mathbb{R}^n$. This implies that $P_X$ is pointed, meaning it does not contain a line. Therefore, it is projectively equivalent to a compact set \cite[Theorem 3.36]{joswig}. Embed $\mathbb{R}^{n+1}$ into $\mathbb{P}^{n+1}$ by the map 
$$
\iota (x_1, \ldots, x_{n+1}) = (1:x_1 : \cdots : x_{n+1}).
$$
Let $l$ be the transformation of $\mathbb{P}^{n+1}$ defined by the following $(n+2) \times (n+2)$ matrix
$$
\begin{bmatrix}
1 & 0 & \cdots & 0 & 1 \\
0 & 2 & 0 & \cdots & 0 \\
\vdots & 0 & \ddots & 0 & \vdots \\
0 & \cdots & 0 & 2 & 0 \\
-1 & 0 & \cdots & 0 & 1
\end{bmatrix}.
$$
Then by \cite[Lemma 7.1]{joswig} the projective transformation $l$ maps $U$ to the sphere $S \subset \mathbb{R}^{n+1}$. The tangential hyperplane at the north pole $(1: 0 : \cdots : 0 : 1)$ is the image of the hyperplane at infinity. Moreover, the topological closure of $l(P_X)$ is a compact convex body so long as the origin is in the interior of $P^*_X$. In this case, we call the convex body $l(P_X)$ the \emph{Voronoi body}. The Voronoi body is full dimensional and contains the origin in its interior. Its polar dual
$$
l(P_X)^\circ := \left\{ y \in \mathbb{R}^{n+1} \ :\ \sum_{i=1}^n x_i y_i \leq 1\text{ for all } x \in \pi(P_X)  \right\}
$$ 
is also full dimensional and has the origin in its interior. If we apply $l^{-1}$ to $l(P_X)^\circ$ we obtain an unbounded polyhedron, which is exactly the Delaunay lift $P^*_X$ of $X$. For more details, see \cite{joswig}.

We now study convergence of Voronoi and Delaunay cells. More precisely, given a real algebraic curve $X$ and a sequence of samplings $A_N \subset X$ with $|A_N| = N$, we show that Voronoi (or Delaunay) cells from the Voronoi (or Delaunay) cells of the $A_N$ limit to Voronoi (or Delaunay) cells of $X$. We begin by introducing two notions of convergence which describe the limits.

The \emph{Hausdorff distance} of two compact sets $B_1$ and  $B_2$ in $\mathbb{R}^n$ is defined as
$$
d_h(B_1,B_2) := \sup \left \{  \adjustlimits \sup_{x \in B_1} \inf_{y \in B_2} d(x,y), \adjustlimits \sup_{y \in B_2} \inf_{x \in B_1} d(x,y) \right \}.
$$
More intuitively, we can define this distance as follows. If an adversary gets to put your ice cream on either set $B_1$ or $B_2$ with the goal of making you go as far as possible, and you get to pick your starting place in the opposite set, then $d_h(B_1,B_2)$ is the farthest the adversary could make you walk in order for you to reach your ice cream.

\begin{definition}
\label{def:hausdorff}
A sequence $\{B_\nu\}_{\nu\in \mathbb{N}}$ of compact sets is \emph{Hausdorff convergent} to $B$ if $d_h(B,B_\nu) \rightarrow 0$ as $\nu \rightarrow \infty$.

\end{definition}
\begin{definition}
\label{def:wijsman}
Given a point $x \in \mathbb{R}^n$ and a closed set $B \subset \mathbb{R}^n$, define
$$
d_w(x,B) = \inf_{b \in B} d(x,b).
$$
A sequence $\{B_\nu\}_{\nu\in \mathbb{N}}$ of compact sets is  \emph{Wijsman convergent} to $B$ if for every $x \in \mathbb{R}^n$, we have that 
$$
d_w(x,B_\nu) \rightarrow d_w(x,B).
$$
\end{definition}

An \emph{$\epsilon$-approximation} of a real algebraic variety $X$ is a discrete subset $A_\epsilon \subset X$ such that for all $y\in X$ there exists an $x \in A_\epsilon$ so that $d(y,x) \leq \epsilon$.
By definition, when $X$ is compact a sequence of $\epsilon$-approximations as $\epsilon$ approaches 0 is Hausdorff convergent to $X$. For all $X$, a sequence of $\epsilon$-approximations as $\epsilon$ approaches 0 is Wijsman convergent to $X$.
We use Wijsman convergence as a variation of Hausdorff convergence which is well suited for unbounded sets. Delaunay cells are always compact, while Voronoi cells may be unbounded. 

We now study convergence of Delaunay cells of $X$, and introduce a condition on real algebraic varieties which ensures that the Delaunay cells are simplices.

\begin{definition}
\label{def:generic}
We say that an algebraic variety $X\subset \mathbb{R}^n$ is \emph{Delaunay-generic} if $X$ does not meet the closure of any maximally inscribed disc at greater than $n+1$ points.
\end{definition}

\begin{example}
 The standard paraboloid $U$ in any dimension $n+2$ is not Delaunay-generic because it contains $n$-spheres. 
\end{example}

Although the focus of this paper is on algebraic curves in $\mathbb{R}^2$, we state the following theorem for curves in $\mathbb{R}^n$ because the proof holds at this level of generality.

\begin{theorem}
\label{thm:delaunay_convergence}
Let $X \subset \mathbb{R}^n$ be a Delaunay-generic compact algebraic curve, and let $\{A_\epsilon\}_{\epsilon \searrow 0}$ be a sequence of $\epsilon$-approximations of $X$. Every maximal Delaunay cell is the Hausdorff limit of a sequence of Delaunay cells of $A_{\epsilon}$. 
\end{theorem}
\begin{proof}
Consider a sequence $\{A_\epsilon\}_{\epsilon \searrow 0}$ of $\epsilon$-approximations of $X$, where $\epsilon \searrow 0$ indicates a decreasing sequence of positive real numbers $\epsilon_\nu$ for $\nu \in \mathbb{N}$. We will study the convex sets
$
P^*_{A_{\epsilon}} = conv(a_U\ |\ a \in A_{\epsilon}),
$
where $a \mapsto a_U = (a,||a||^2)$ lifts $a$ to the paraboloid $U$. 
The lower faces of $P^*_{A_{\epsilon}}$ project to Delaunay cells of $A_\epsilon$ \cite[Theorem 6.12, Theorem 7.7]{joswig}.

We now apply \cite[Theorem 3.5]{nidhi} to our situation.
This result says the following.
Let $C$ be a curve and $B_\epsilon$ be a sequence of $\epsilon$-approximations of $C$.
Suppose every point on $C$ which is contained in the boundary of $conv(C)$ is an extremal point of $conv(C)$, meaning it is not contained in the open line segment joining any two points of $conv(C)$. Let $F$ be a simplicial face of $conv(C)$ which is an exposed face of $conv(C)$ with a unique exposing hyperplane.
Then $F$ is the Hausdorff limit of a sequence of facets of $conv(B_\epsilon)$. We apply this result in the case when $C= X_U = \{x_U \in \mathbb{R}^{n+1} \ |\ x \in X\}$ and $B_\epsilon = (A_{\epsilon})_U = \{a_U \in \mathbb{R}^{n+1}\ |\ a \in A_{\epsilon}\}$. 

Since every point on $U$ is extremal in $conv(U)$ and $conv(X_U) \subset conv(U)$, every point on $X_U$ which is contained in the boundary of $conv(X_U)$ is also extremal in $conv(X_U)$. 
A maximal Delaunay cell of $X$ is a simplex because $X$ is Delaunay-generic.
Consider a maximal Delaunay cell of $X$ which is not a vertex. It has a unique description as $Del_X(B)$ for a disc $B$.
Proposition \ref{prop:delaunay_faces} establishes a one-to-one correspondence between such Delaunay cells and lower exposed faces of $P^*_X$, which are uniquely exposed by the hyperplane containing $(\partial \overline{B})_U$.
 In this case, \cite[Theorem 3.5]{nidhi} holds, so the result is proved.

If a maximal Delaunay cell is a vertex, then it is a point $x \in X$. It is then also an extremal point of $conv(X_U)$. Since $conv(B_\epsilon)$ is sequence of compact convex sets converging in the Hausdorff sense to $P_X^*$, by \cite[Lemma 3.1]{nidhi} there exists a sequence of points of $B_\epsilon$ which are extremal points of $conv(B_\epsilon)$ converging to $x_U$. So, their projections are Delaunay cells of $A_\epsilon$ converging to $x$, since every point in a finite point set is a Delaunay cell of that point set. 
\end{proof}

We will now study limits of Voronoi cells, using results from \cite{skeletons}, which studies convergence of Voronoi cells of \emph{r-nice sets} (for a definition, see \cite[p. 119]{skeletons}). In the plane, these are open sets whose boundary satisfies some properties. In particular, open sets whose boundaries are an algebraic curve with positive reach $r$ satisfy the $r$-nice condition. All closed $C^2$ submanifolds of $\mathbb{R}^n$ have positive reach, and thus in particular, a compact smooth algebraic curve in $\mathbb{R}^2$ has positive reach; we refer the reader to \cite{scholtes} for further discussion of which sets have positive reach. 

To study continuity and convergence of closed sets in the plane, Brandt uses the \emph{hit-miss topology} on the set $\mathcal{F}$ of closed subsets of the plane \cite[Section 1-2]{matheron}. 
\begin{definition}
\label{def:semicontinuous}
In the hit-miss topology, a sequence $\{F_n\}$ \emph{converges} to $F$ if and only if
\begin{enumerate}
    \item \label{def:upper} for any $p \in F$, there is a sequence $p_n \in F_n$ such that $p_n \rightarrow p$; and
    \item \label{def:lower}
    if there exists a subsequence $p_{n_k} \in F_{n_k}$  converging to a point $p$, then $p \in F$.
\end{enumerate}
Then, to determine if a function with range in $\mathcal{F}$ is continuous, we need to examine the above conditions for sequences of sets obtained by applying the function to countable convergent sequences in $\mathbb{R}^2.$ If all such sequences satisfy (\ref{def:upper}) then the function is \emph{upper-semicontinuous}. If all such sequences satisfy (\ref{def:lower}) then it is \emph{lower-semicontinuous}. If a function satisfies both then it is continuous.
\end{definition}

\begin{lemma}
\label{lem:voronoi}
Let $X \subset \mathbb{R}^2$ be a smooth algebraic plane curve. Then the function $Vor_X: X \rightarrow \mathcal{F}$ sending $x \mapsto Vor_X(x)$ is continuous in the hit-miss topology.
\end{lemma}
\begin{proof}
By \cite[Theorem 2.2]{skeletons}, the Voronoi function $Vor_X:X \rightarrow \mathcal{F}$ is lower semicontinuous. By \cite[Theorem 3.2]{skeletons}, if the curve is $C^2$ and the \emph{skeleton} (locus of centers of maximally inscribed discs) is closed, then the Voronoi function is continuous. A smooth algebraic curve is $C^2$. The skeleton is closed because a smooth curve satisfies the $r$-nice condition, and $r$-nice curves have closed skeletons \cite{skeletons}.
\end{proof}

By \cite[p.10]{matheron}, convergence in the hit-miss topology is equivalent to Wijsman convergence. In what follows, we rephrase the results from \cite{skeletons} in the setting of Wijsman convergence of Voronoi cells of plane curves, and extend them to singular curves.

\begin{theorem}
\label{thm:voronoi_convergence}
Let $X$ be a compact smooth algebraic curve in $\mathbb{R}^2$ and $\{A_\epsilon\}_{\epsilon \searrow 0}$ be a sequence of $\epsilon$-approximations of $X$. Every Voronoi cell is the Wijsman limit of a sequence of Voronoi cells of $A_\epsilon$.
\end{theorem}

\begin{proof}
By Lemma \ref{lem:voronoi}, the function $Vor_X: X \rightarrow \mathcal{F}$ is continuous. Theorem 3.1 from \cite{skeletons} states that in this case, if $x_\epsilon$ is a sequence such that $x_\epsilon \in A_\epsilon$ and $x_\epsilon \rightarrow x$, then $Vor_{A_\epsilon}(x_\epsilon) \rightarrow Vor_X(x)$. Such a sequence must exist because for all $y \in X$, there exists a $y_\epsilon \in A_\epsilon$ such that $d(y,y_\epsilon) \leq \epsilon$.
\end{proof}

We now investigate the structure of Voronoi cells of different types of singular points. We show examples in which the Voronoi cell at a singular point is $0$-, $1$-, and $2$- dimensional in Proposition \ref{prop:singular_points} and Figure~\ref{fig:singular_varieties}. First, we need a glueing lemma. 

\begin{lemma}
\label{lem:glueing}
Let $C$ and $D$ be subsets of $\mathbb{R}^2$ containing a point $p \in C \cap D$. Then 
$$
Vor_{C \cup D}(p) = Vor_C(p) \cap Vor_D(p).
$$
\end{lemma}
\begin{proof}
A point $x \in Vor_{C \cup D}(p)$ is closer to $p$ than it is to any other point of $C$ or $D$. On the other hand, a point in $Vor_C(p) \cap Vor_D(p)$ is closer to $p$ than it is to any other point of $C$ or $D$.
\end{proof}

\begin{figure}[h]
\centering
\begin{subfigure}{.45\textwidth}
  \centering
  \includegraphics[width=.9\linewidth]{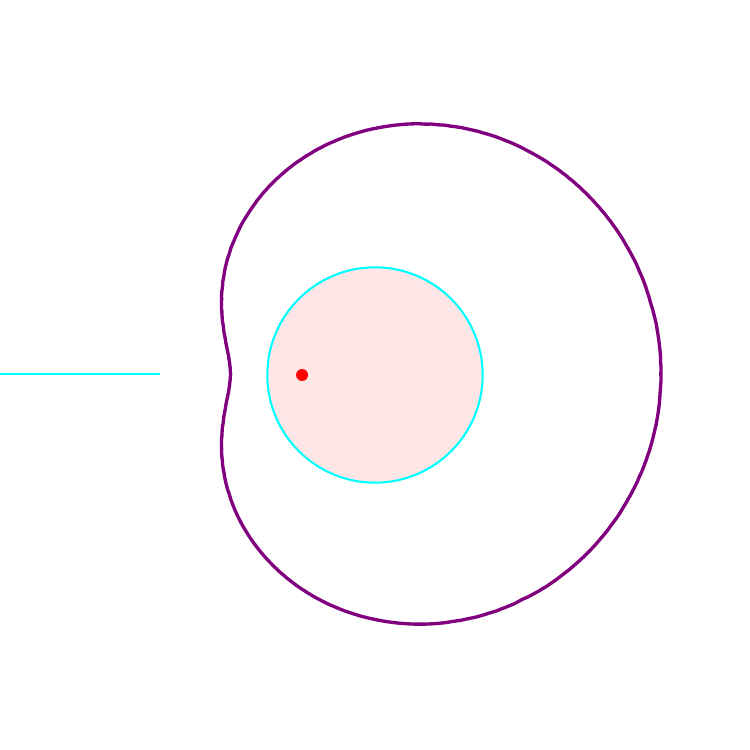}
  \caption{$(x^2 + y^2 - x)^2 - (1.5)^2 (x^2 + y^2) = 0$\\ An isolated singularity and its\\
  2-dimensional Voronoi cell. }
\end{subfigure}%
\begin{subfigure}{.45\textwidth}
  \centering
  \includegraphics[width=.9\linewidth]{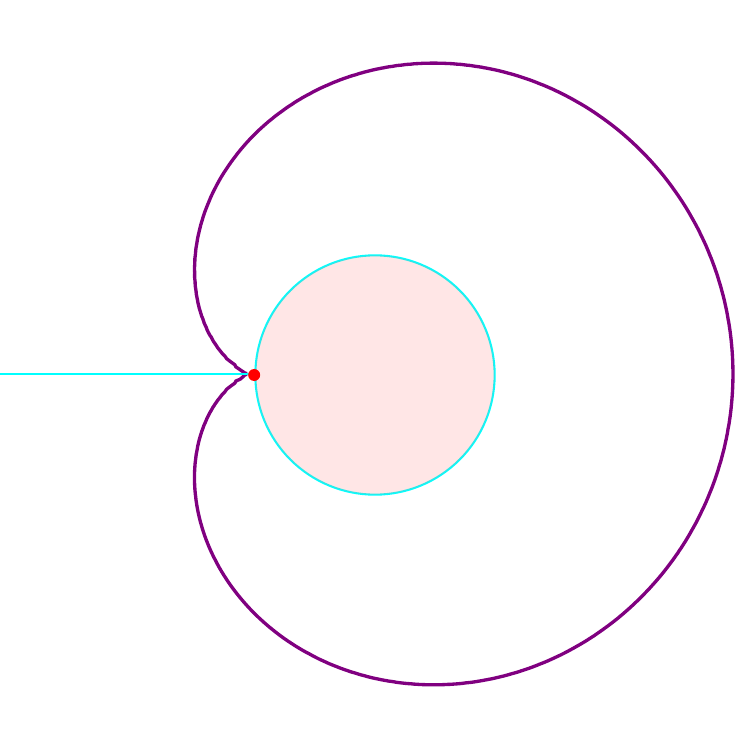}
  \caption{$(x^2 + y^2 - x)^2 -  (x^2 + y^2)= 0$\\
  A cusp and its 2-dimensional Voronoi cell. }
\end{subfigure}
\begin{subfigure}{.45\textwidth}
  \centering
  \includegraphics[width=.9\linewidth]{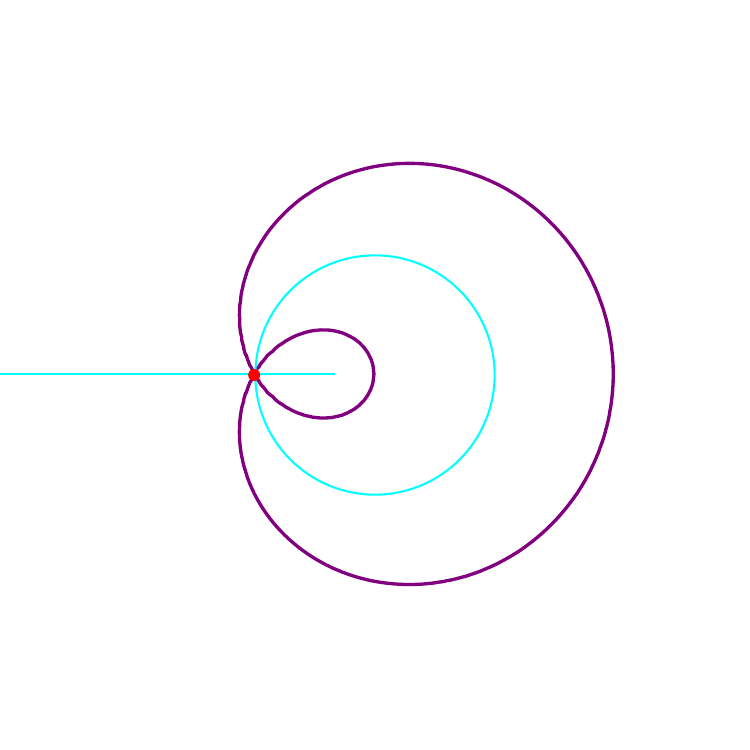}
  \caption{$(x^2 + y^2 - x)^2 - (0.5)^2 (x^2 + y^2)= 0$\\ A node and its 0-dimensional Voronoi cell. }
\end{subfigure}%
\begin{subfigure}{.45\textwidth}
  \centering
  \includegraphics[width=.8\linewidth]{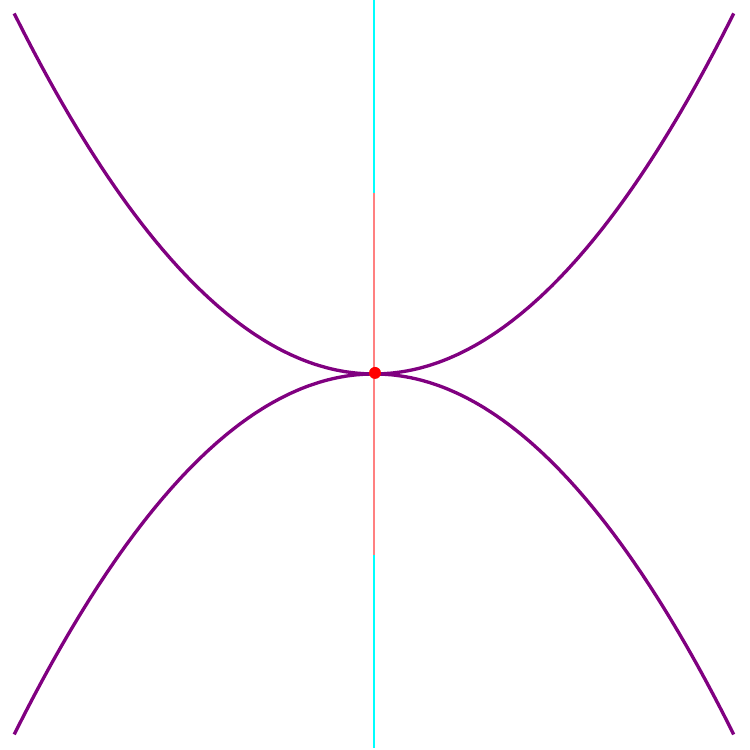}
  \caption{$x^4 - y^2= 0$\\
  A tacnode and its 1-dimensional Voronoi cell. }
\end{subfigure}
\caption{Four singular varieties with singular point $(0,0)$. In each case the medial axis is blue, the singular point is red, and its Voronoi cell is pink.}
\label{fig:singular_varieties}
\end{figure}

\begin{proposition}
\label{prop:singular_points}
Let $X \subset \mathbb{R}^2$ be a real algebraic plane curve and $p$ be a singular point.
\begin{enumerate}
    \item If $p$ is a node, then its Voronoi cell is 0-dimensional and equal to $p$;
    \item If $p$ is a tacnode, then its Voronoi cell is 1-dimensional.
    \item If $p$ is an isolated point, its Voronoi cell is 2-dimensional;
\end{enumerate}
\end{proposition}
\begin{proof}
\hfill

\begin{enumerate}
\item If $p$ is a node, then we claim that the only point contained in $Vor_X(p)$ is $p$. At $p$, the curve meets in two branches which have distinct tangent directions at $p$. If we treat this as two separate 1-dimensional subsets $X_1$ and $X_2$ and apply Lemma \ref{lem:glueing}, we see that $Vor_X(p) = Vor_{X_1}(p) \cap Vor_{X_2}(p)$. But, since $p$ is a smooth point of $X_1$ and $X_2$, the Voronoi cells $Vor_{X_1}(p)$ and $Vor_{X_2}(p)$ are each contained in their respective normal directions, which are distinct. Therefore, $Vor_{X_1}(p) \cap Vor_{X_2}(p) = p$.
    
\item If $p$ is a tacnode, two branches of the curve meet at $p$ where they share a tangent direction, and thus a normal line. We can choose $\epsilon>0$ so that we can separate $X \cap B(p, \epsilon)$ into subsets $X_1$ and $X_2$ corresponding to the two branches. Both $Vor_{X_1}(p)$ and $Vor_{X_2}(p)$ are $1$-dimensional subsets of the same normal line. By Lemma \ref{lem:glueing}, $Vor_X(p)$ is a $1$-dimensional subset of this normal line. 

\item Suppose $p$ is an isolated point. Then there is a ball $B(p,r)$ centered at $p$ such that the ball contains no other points of the curve $X$. Therefore, the ball $B(p,r/2)$ is entirely contained in $Vor_X(p)$, so it is $2$-dimensional.
\end{enumerate}
\end{proof}

\begin{example}
\label{ex:cusp_converge}
In this example we illustrate why Theorem \ref{thm:voronoi_convergence} fails when the curve has a singular point. From this example it will be clear that the singular points must be included in the samples $A_\epsilon$, and it turns out that this condition is enough to extend Theorem \ref{thm:voronoi_convergence} to the singular case.

Consider the curve defined by the equation $y^2 = x^3$. In \cite[Remark 2.4]{voronoi} the authors give equations for the Voronoi cell of the cusp at the origin. This region is
$$
Vor_{y^2=x^3}((0,0)) = \{(x,y) \in \mathbb{R}^2 \ : \ 
27 y^4 + 128 x^3 + 72 x y^2 + 32 x^2 + y^2 + 2 x \leq 0
\}.
$$
In Figure \ref{fig:cuspidal} we give three  $\epsilon$-approximations of the curve and the corresponding Voronoi decompositions. Let $\epsilon = 1/n$. The points in the $\epsilon$-approximation $A_{\epsilon}$ are given by:
$$
A_{\epsilon} = \left \{\left (\frac{j}{n}, \pm \left (\frac{j}{n}\right)^{3/2} \right )\right \}_{j = 1}^{\infty}.
$$
As we can see in Figure \ref{fig:cuspidal}, there is no sequence of cells converging to $Vor_{y^2=x^3}((0,0))$ because the $x$-axis, present due to the symmetrical nature of the sample, always divides the Voronoi cell.

\begin{figure}[h]
    \centering
    \includegraphics[width= \linewidth]{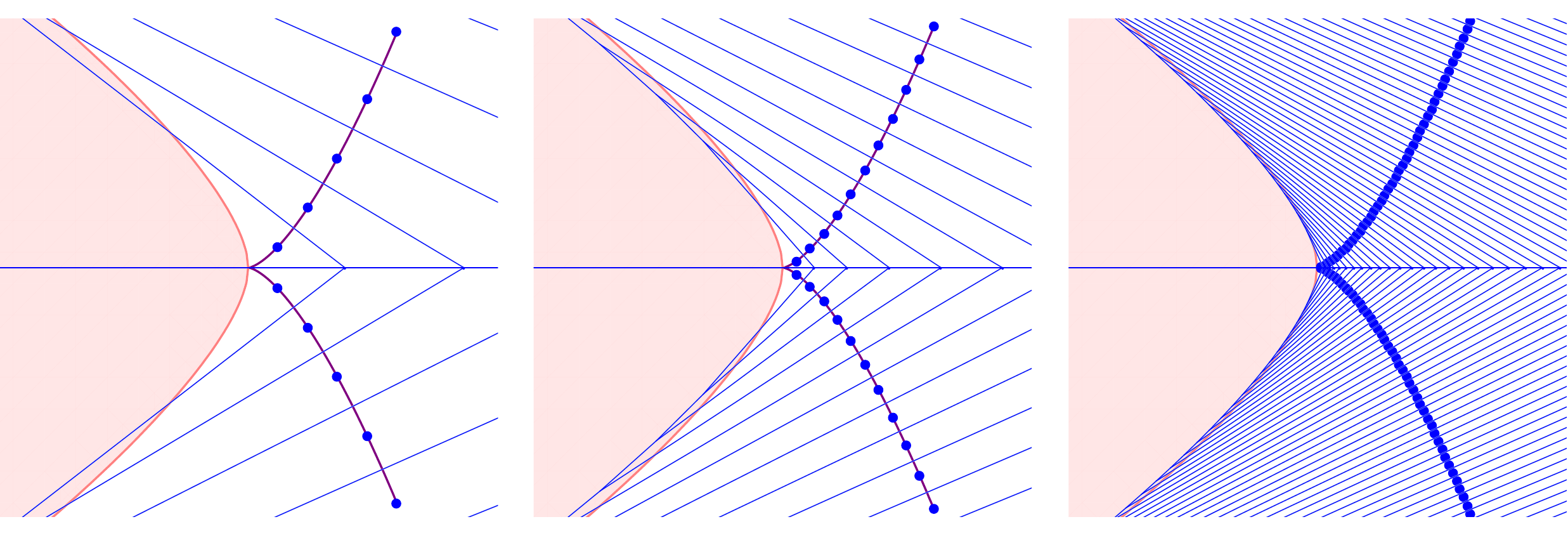}
    \caption{Some $\epsilon$-approximations of the curve $y^2=x^3$ and their Voronoi diagrams. The Voronoi cell of the cusp $(0,0)$ is shown in pink. This figure is discussed in Example \ref{ex:cusp_converge}.}
    \label{fig:cuspidal}
\end{figure}
\end{example}

We now are able to expand Theorem \ref{thm:voronoi_convergence} to include singular varieties.

\begin{proposition}
\label{prop:singular_convergence}
Let $X \subset \mathbb{R}^2$ be a compact algebraic curve and $\{A_\epsilon\}_{\epsilon \searrow 0}$ a sequence of $\epsilon$-approximations with the singular locus $Sing(X) \subset A_\epsilon$ for all $\epsilon$. Then every Voronoi cell of $X$ is the Wijsman limit of a sequence of Voronoi cells of $A_\epsilon$.
\end{proposition}

\begin{proof}
By \cite[Theorem 2.2]{skeletons}, the Voronoi function is always lower-semicontinuous. So, we must show that condition (\ref{def:upper}) in Definition \ref{def:semicontinuous} holds. That is, we need that for all $p \in X$, there is a sequence $p_\epsilon \in A_\epsilon$ with $p_\epsilon \rightarrow p$ such that for any $x \in Vor_X(p)$ there is an $x_\epsilon \in Vor_{A_\epsilon}(p_\epsilon)$ with $x_\epsilon \rightarrow x$. We distinguish the cases when $p$ is smooth and singular.

If $p$ is a smooth point on $X$, and $x \in Vor_X(p)$, there exists an $\epsilon$ such that $x$ and $p$ are both in the Voronoi cell $Vor_{A_\epsilon}(p_\epsilon)$ for some $p_\epsilon$.

Suppose now that $p \in X$ is a singular point. We wish to show that there is a sequence of Voronoi cells converging to $Vor_X(p)$, and we take the sequence $Vor_{A_\epsilon}(p)$. To establish convergence, it is now enough to show that for all $x \in Vor_X(p)$, there is an $x_\epsilon \in Vor_{A_\epsilon}(p)$ with $x_\epsilon \rightarrow x$. Since $x \in Vor_X(p)$, we have that $x$ is closer to $p$ than it is to any other point in $X$. So, in particular, $x \in Vor_{A_\epsilon}(p)$. 

Now we have shown that for each $p \in X$, condition (\ref{def:upper}) in Definition \ref{def:semicontinuous} holds. Therefore, for each $p \in X$, we have sequences of Voronoi cells which are convergent to $Vor_p(X)$ in the hit-miss topology. Since convergence in the hit-miss topology and Wijsman convergence are equivalent, every Voronoi cell of $X$ is the Wijsman limit of a sequence of Voronoi cells of the $A_\epsilon$.
\end{proof}

This concludes the proof of Theorem \ref{thm:convergence}.

\section{Medial Axis}\label{sec:medial_axis} 

Let $X = V(F)\subset \mathbb{R}^2$ be a smooth algebraic plane curve.
We now study the medial axis of $X$, as defined in Definition \ref{def:medial}. 
The Zariski closure of the medial axis is an algebraic variety which has the same dimension as the medial axis. We can obtain equations in variables $x,y$ for the ideal $I$ of a variety containing the Zariski closure of the medial axis in the following way.

Let $(s,t)$ and $(z,w)$ be two points on $X$. Then, $s,t,z,$ and $ w$ satisfy the equations 
\[ F(s,t)=0\ \text{and } \  F(z,w)=0 . \] 
If $(x,y)$ is equidistant from $(s,t)$ and $(z,w)$ then
\[ (x-s)^2+(y-t)^2=(x-z)^2+(y-w)^2 . \] 
Furthermore, $(x,y)$ must be a critical point of the distance function from both $(s,t)$ and $(z,w)$. Thus we require that the determinants of the following $2 \times 2$ augmented Jacobian matrices vanish: 
\begin{equation*} 
\begin{bmatrix}
    x-s & y-t\\
    F_s & F_t  
\end{bmatrix} ,
\ \ \ \ \  
\begin{bmatrix}
    x-z & y-w\\
    F_z & F_w  
\end{bmatrix} ,
\end{equation*}
where $F_s,F_t,F_z$ and $F_w$ denote the partial derivatives of $F(s,t)$ and $F(z,w)$, respectively. 
Let
\begin{align*}
    I =  \langle & 
F(s,t), F(z,w), (x-s)^2+(y-t)^2-(x-z)^2-(y-w)^2,\\
 & (x-s)F_t - (y-t)F_s , (x-z)F_w - (y-w)F_z \rangle.
\end{align*}

Then,
$
J = \left (  I : (s-z,t-w)^\infty \right ) \cap \mathbb{R}[x,y]
$
is an ideal whose variety contains the Zariski closure of the medial axis.



We now study the medial axis from the perspective of Voronoi cells.
It has been observed that an approximation of the medial axis arises as a subset of the Voronoi diagram of finitely many points sampled densely from a curve \cite{medial_approximation}. We now discuss theoretical results given in \cite{skeletons} about the convergence of medial axes.
Let $X$ be a compact smooth algebraic plane curve, and let $A_\epsilon$ be an $\epsilon$-approximation of $X$. A Voronoi cell $Vor_{A_\epsilon}(a_\epsilon)$ for $a_\epsilon \in A_\epsilon$ is polyhedral, meaning it is an intersection of half-spaces.

\begin{definition}
\label{def:short_long_edges}
 For sufficiently small $\epsilon$, exactly two edges of $Vor_{A_\epsilon}(a_\epsilon)$ will intersect $X$ \cite{skeletons}. We call these edges the \emph{long edges} of the Voronoi cell, and all other edges are called \emph{short edges}. An example is given in Figure \ref{fig:short_long}.
\end{definition}

\begin{figure}[h]
    \centering
    \includegraphics[height=2.5in]{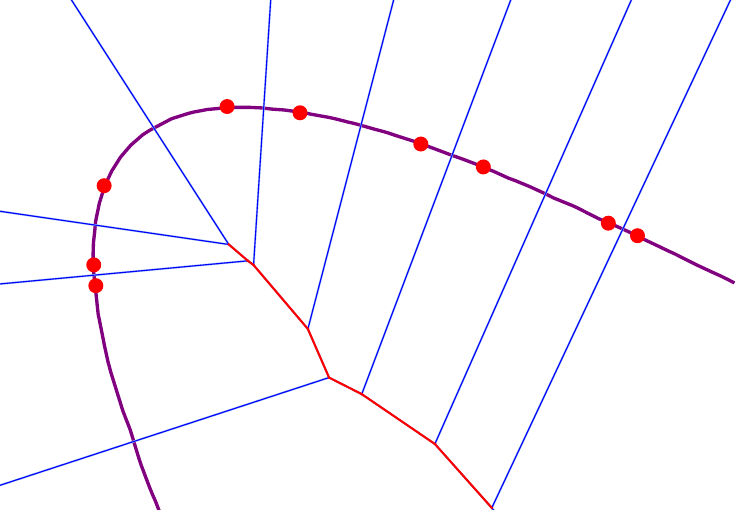}
    \caption{The long edges (blue) and short edges (red) of Voronoi cells of points sampled from the butterfly curve as in Definition \ref{def:short_long_edges}.}
    \label{fig:short_long}
\end{figure}

In this case, let $\hat{S}_\epsilon(a_\epsilon)$ denote the union of the short edges and vertices of the Voronoi cell $Vor_{A_\epsilon}(a_\epsilon)$. An \emph{$\epsilon$-medial axis approximation} is the set of all short edges
$$
\hat{S}_\epsilon = \bigcup_{p \in A_\epsilon} \hat{S}_\epsilon(p).
$$

\begin{proposition}(\cite[Theorem 3.4]{skeletons})\label{prop:brandt_medial}
Let $X$ be a compact smooth algebraic plane curve. The medial axis approximations $\hat{S}_\epsilon$ converge to the Euclidean closure of the medial axis.
\end{proposition}

The following corollary shows the relationship between the medial axis, parts of Voronoi diagrams, and maximally inscribed circles for $\epsilon$-approximations.

\begin{corollary}
Let $\{A_\epsilon\}_{\epsilon \searrow 0}$ be a sequence of $\epsilon$-approximations of a compact smooth algebraic curve ${X \in \mathbb{R}^2}$.
\label{cor:medial_approx}
\begin{enumerate}
    \item The collection of vertices of the Voronoi diagrams of the $A_\epsilon$ converge to the medial axis. 
    \item The collection of centers of maximally inscribed discs of the $A_\epsilon$ converge to the medial axis.
\end{enumerate}
\end{corollary}
\begin{proof}
This is a consequence of Theorem \ref{thm:delaunay_convergence}, Theorem \ref{thm:voronoi_convergence}, and Proposition \ref{prop:brandt_medial}.
\end{proof}

\begin{example}
In Figure \ref{fig:trott_medial} we display the centers of maximally inscribed circles, or equivalently circumcenters of the Delaunay triangles, for an $\epsilon$-approximation of the butterfly curve where 898 points were sampled. In Figure \ref{fig:butterfly_v_medial} we show the short edges of Voronoi cells from an $\epsilon$-approximation of the butterfly curve where 101 points were sampled.
\end{example}

\begin{figure}[h]
\centering
\begin{minipage}{.47\textwidth}
  \centering
  \includegraphics[width=\linewidth]{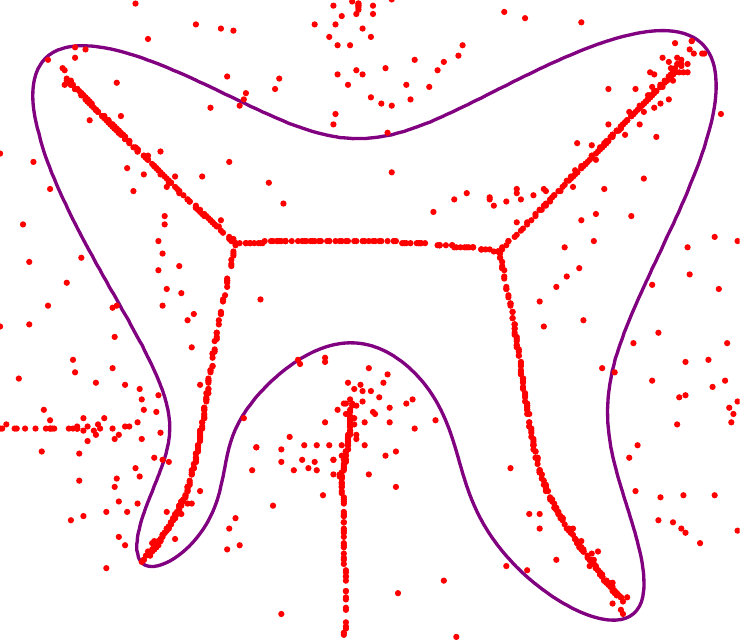}
  \captionof{figure}{A medial axis approximation of the butterfly curve obtained from circumcenters of Delaunay triangles, which are shown in red.}
  \label{fig:trott_medial}
\end{minipage}%
\begin{minipage}{.47\textwidth}
  \centering
  \includegraphics[width=\linewidth]{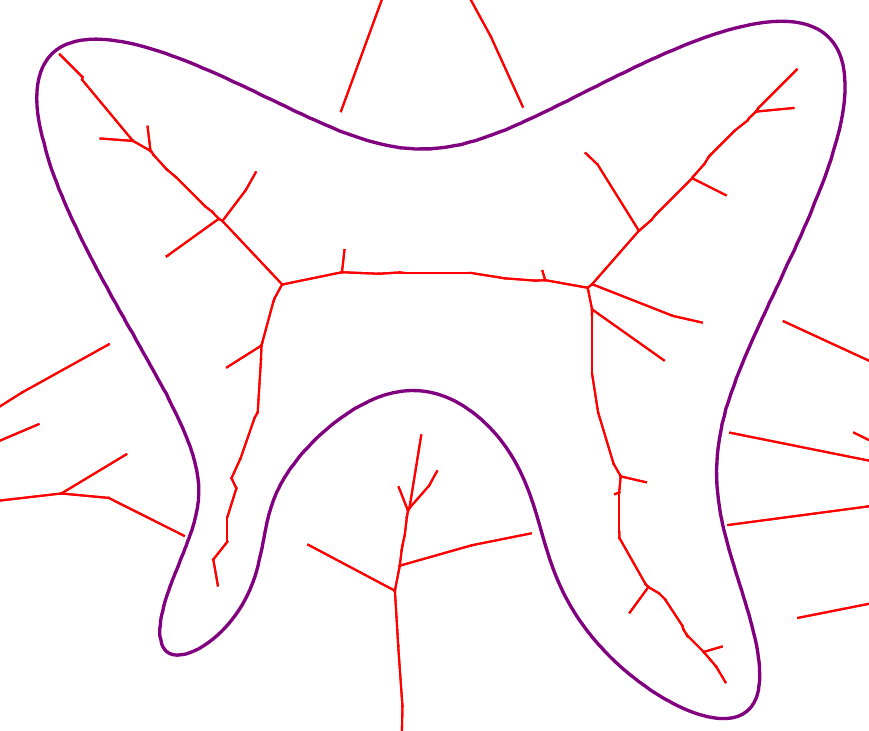}
  \captionof{figure}{A medial axis approximation of the butterfly curve obtained from short edges of Voronoi cells, which are shown in red.}
  \label{fig:butterfly_v_medial}
\end{minipage}
\end{figure}


The medial axis plays an important role in applications for understanding the connected components and regions of a shape. As such, it is a very well-studied problem in computational geometry to find approximations of the medial axis from point clouds. 
A survey on medial axis computation is given in \cite{medial_survey}.


\section{Curvature and the Evolute}\label{sec:evolute}

Curvature of plane curves, osculating circles and evolutes have interested mathematicians since antiquity. 
We refer readers to works of Salmon in the 19th century \cite{Salmon1, Salmon2} and to modern lectures by Fuchs and Tabachnikov outlining this history \cite[Chapter 3]{Fuchs}.


We now discuss the minimal radius of curvature of a plane curve. This is one of the two quantities which determines the reach, see Equation \ref{eqn:reach_min}. 

\begin{definition}\label{def:curvature}
Let $X\subset \mathbb{R}^2$ be an algebraic curve and $p \in X$ be a smooth point that is not a point of inflection. The \emph{osculating circle} at $p$ is the circle that passes through $p$ and an additional pair of points infinitesimally close to $p$. The \emph{center of curvature} at $p$ is the center of the osculating circle at $p$. The \emph{radius of curvature} at $p$ is the distance from $p$ to its center of curvature. 
\end{definition}

An alternative definition of center and radius of curvature can be given using envelopes. 

\begin{definition}
The \emph{envelope} of a one-parameter family of algebraic plane curves given implicitly by 
${F(x,y,t)=0}$
is a curve that touches every member of the family tangentially. The envelope is the variety defined by the ideal 
\[ \left \langle \frac{\partial F}{\partial t}, \  F(x,y,t) \right \rangle \cap \mathbb{R}[x,y]. \] 
\end{definition}

The envelope of the family of normal lines parametrized by the points of the curve is called its \emph{evolute}. Equivalently, the evolute is the locus of the centers of curvature.
A generalization of the evolute to all dimensions is called the \emph{ED discriminant}, and is studied in \cite{ed}.
The article shows that for general smooth algebraic plane curves, the degree of the evolute is
$3d(d-1)$ \cite[Example 7.4]{ed}. See also \cite[p.95-96]{Salmon2} for a discussion of the degree of the evolute.

We now derive a formula for the square of the radius of curvature of a plane curve at a point. Our derivation follows Salmon \cite[p.84-98]{Salmon2}. 

\begin{proposition}\cite[p.84-86]{Salmon2}
Let $X=V(F(x,y)) \in \mathbb{R}^2$ be a smooth curve of degree $d$. The square of the radius of curvature at a point $(x_0,y_0) \in V(F)$ that is not a point of inflection is given by following expression in the partial derivatives of $F$ evaluated at $(x_0,y_0)$:
\begin{equation}\label{urc} 
    R^2= \frac{ (F_x^2+F_y^2)^3}{(F_{xx}F_{y}^2-2F_{xy}F_xF_y+F_{yy} F_x^2)^2}. 
\end{equation}
\end{proposition}

\begin{proof}
The equation of a normal line to $X$ at a point $(x,y) \in X$  in the variables $(\alpha, \beta)$ is  
\begin{equation}\label{nl}
    F_y(\alpha-x)-F_x(\beta-y)=0. 
\end{equation}

The total derivative of the equation for the normal line is 
\begin{equation}\label{td1}
    \left(F_{xy}+F_{yy}\frac{dy}{dx}\right)(\alpha-x)-\left(F_{xx}+F_{xy}\frac{dy}{dx}\right)(\beta-y)-F_y+F_x\frac{dy}{dx}=0. 
\end{equation}

The total derivative of $F(x,y)$ is 
\begin{equation}\label{td2}
    F_x(x,y)+F_y\frac{dy}{dx} = 0.
\end{equation}

 The equations \eqref{nl}, \eqref{td1} are a system of two linear equations in the unknowns $\{ \alpha, \beta \}$. We solve this system to obtain expressions for $\alpha$ and $\beta$ in terms of $x$, $y$, and $\frac{dy}{dx}$. We substitute in for $\frac{dy}{dx}$ the expression given by \eqref{td2}. The center of curvature of $X$ at a point $(x,y) \in X$ is given by the coordinates $(\alpha, \beta)$, which are now expressions in $x$ and $y$. 
 
 The squared radius of curvature at a point $(x,y)$ is its squared distance to its center of curvature $(\alpha, \beta)$, so we have 
$ R^2=(\alpha-x)^2+(\beta-y)^2$.
Substituting in the equations for $\alpha$ and $\beta$, we find 
 \begin{equation*}
 R^2= \frac{ (F_x^2+F_y^2)^3}{(F_{xx}F_{y}^2-2F_{xy}F_xF_y+F_{yy} F_x^2)^2}. 
 \end{equation*}
 We note that the denominator evaluates to zero only at points of inflection.
 \end{proof}


\begin{definition}
The \emph{degree of critical curvature} of a smooth algebraic curve $X \subset \RR^2$ is the degree of the variety obtained by intersecting the Zariski closure $\overline{X} \subset \mathbb{P}_{\mathbb{C}}^2$  with the variety of the total derivative of the equation for the squared radius of curvature. If $X \subset {\mathbb{R}}^2$ is a smooth, irreducible algebraic curve of degree greater than or equal to $3$, then the intersection consists of finitely many points, called the \emph{points of critical curvature}. Thus the degree of critical curvature of $X$ gives an upper bound for the number of real points of critical curvature of $X$. 
\end{definition}

\begin{remark}
In the differential geometry literature, points of critical curvature are called \textit{vertices}.
\end{remark}

\begin{theorem}[\cite{Salmon2}, p.97]
Let $X \subset {\mathbb{R}}^2$ be a smooth, irreducible algebraic curve. Then the degree of critical curvature of $X$ is $6d^2-10d$. 
\end{theorem}

We remark that the critical points of curvature of $X$ give cusps on the evolute \cite[p.97]{Salmon2}. 
That is, if a normal line is drawn through a point of critical curvature on a curve, then the normal line will pass through a cusp of the evolute. 
In addition, the evolute of a curve of degree $d$ has $d$ cusps at infinity \cite[p.95]{Salmon2}. Thus the evolute of a general plane curve of degree $d$ has $6d^2-10d+d=6d^2-9d$ cusps \cite[p.97]{Salmon2}. In Figure \ref{fig:evolute}, we picture the evolute, the butterfly curve, and the pairs of critical curvature points on the butterfly curve with their corresponding cusp on the evolute.

\begin{figure}[h]
    \centering
    \includegraphics{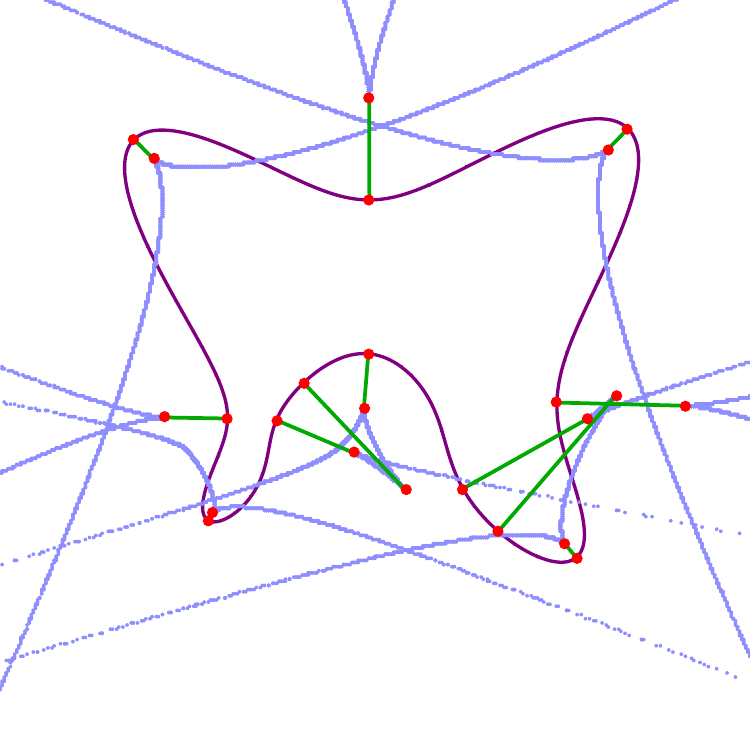}
    \caption{The twelve real points of critical curvature on the butterfly curve, computed in Example \ref{ex:curvature}, joined by green line segments to their centers of curvature. These give cusps on the evolute, which is pictured here in light blue.}
    \label{fig:evolute}
\end{figure}

\begin{example}
\label{ex:curvature}
Consider the butterfly curve. Using the above description, we can compute the 56 points of critical curvature using \texttt{JuliaHomotopyContinuation} \cite{julia}. Twelve of these points are real, and they are plotted in Figure \ref{fig:evolute}. The maximal curvature is approximately $9.65$. This is achieved at the lower left wing of the butterfly.
\end{example}


We now describe how to recover the curvature at a point from the Voronoi cells of a subset of a curve $X$. In applications, Voronoi-based methods are used for obtaining estimates of curvature at a point. 
An overview of techniques for estimating curvature of a variety from a point cloud is given in \cite{feature_detection}. Further, there are also  Delaunay-based methods for estimating curvature of a surface in three dimensions \cite{delaunay_curvature}. 

\begin{theorem}
\label{thm:curvature}
Let $X$ be a smooth, irreducible plane curve of degree at least $3$ and $p\in X$ a point that is not a critical point of curvature. 
Let $\delta$ be less than the distance to the critical point of curvature nearest to $p$, and
let $B(p, \delta)$ be a ball of radius $\delta$ centered at $p$. Then
\begin{enumerate}
    \item
    \label{thm:1pt1}
    The Voronoi cell $\text{Vor}_{X \cap B(p, \delta)}(p)$ is a ray. The distance from $p$ to the endpoint of this ray is the radius of curvature of $X$ at $p$.
    \item
    \label{thm:1pt2}
    Consider a sequence of $\epsilon$-approximations $A_\epsilon$ of $X \cap B(p,\delta)$. Let $a_\epsilon$ be a point such that $p \in Vor_{A_\epsilon}(a_\epsilon)$, and let $d_\epsilon$ be the minimum distance from $a_\epsilon$ to a vertex of $Vor_{A_\epsilon}(a_\epsilon)$. Then, the sequence $d_\epsilon$ converges to the radius of curvature of $X$ at $p$.
\end{enumerate}
\end{theorem}
\begin{proof}
The Voronoi cell $\text{Vor}_{X \cap B(p, \delta)}(p)$ is a subset of the line normal to $X \cap B(p, \delta)$ at $p$. This line has an endpoint either at the center of curvature of $p$ or at a point where it intersects the normal space of a distinct point $p'$ in $X\cap B(p,\delta)$. The point where the normals at $p$ and $p'$ intersect is contained in the Voronoi cell with respect to $X \cap B(p, \delta)$ of each of them, so in particular $X \cap B(p,\delta)$ has a nonempty medial axis. This medial axis has an endpoint which corresponds to a point of critical curvature \cite{medial_endpoints}. This contradicts the constraint on $\delta$. Therefore, the endpoint of the Voronoi cell $\text{Vor}_{X \cap B(p, \delta)}(p)$ is the center of curvature of $p$. This concludes the proof of (\ref{thm:1pt1}).

For (\ref{thm:1pt2}), we know that the sequence $Vor_{A_\epsilon}(a_\epsilon)$ is Wijsman convergent to $Vor_{X \cap B(p, \delta)}(p)$ by Theorem \ref{thm:voronoi_convergence}. Denote by $V_\epsilon$ the set of vertices of $Vor_{A_\epsilon}(a_\epsilon)$. By Corollary \ref{cor:medial_approx}, we also have that the sets $V_\epsilon$ are Wijsman convergent to the endpoint of $Vor_{X \cap B(p,\delta)}(p)$, which we call $p'$. By the definition of Wijsman convergence, this means that for any $x \in \mathbb{R}^2$, $d_w(x,V_\epsilon) \rightarrow d_w(x,p')$. By the definition of $d_w$, we have $d_w(p,V_\epsilon) = d_\epsilon$ and $d_w(p,p')$ is the radius of curvature of $p$. This concludes the second part of the proof.
\end{proof}

The evolute $E$ of a plane curve is the locus of all centers of curvature of the curve. Therefore, to find the evolute using Voronoi cells we may splice the curve into sections and apply Theorem~\ref{thm:curvature}.
Let $X$ be compact and irreducible of degree greater than or equal to $3$. Let $C \subset X$ denote the points of locally maximal curvature and $E_C$ denote the centers of curvature corresponding to points in $C$. Then $X \backslash C$ consists of finitely many components $X \backslash C= X_1\cup \cdots \cup X_n$. Let $\tau$ denote the reach of $X$, and cover each $X_i$ by balls $B_{i,j}$ of radius less than $\tau$.
Let $E_{i,j}$ denote the collection of vertices of Voronoi cells of $X_i \cap B_{i,j}$. Then by Theorem \ref{thm:curvature}, $E \backslash E_C= \cup_{i,j} \overline{E_{i,j}}$. 
Furthermore, for $\epsilon$-approximations $A_{\epsilon,i,j}$ of $X_i \cap B_{i,j}$, the union over $i,j$ of their Voronoi vertices will converge to $E \backslash E_C$ by Theorem \ref{thm:voronoi_convergence}. To find the evolute $E$ one need only to add the finite set of points $E_C$. 


\section{Bottlenecks}\label{sec:bottlenecks}


As in the colloquial sense of the word, a bottleneck refers to a narrowing of a variety, or a place where it gets closer to self-intersection. 
Consider a smooth algebraic variety $X \subset \RR^n$. We define $a\perp b$ by $\sum_{i=1}^na_ib_i=0$ for
$a=(a_1,\dots,a_n),$ $b= (b_1,\dots,b_n) \in \RR^n$.
For a point $x \in X$,
let $(T_xX)_0$ denote the embedded tangent space of $X$ translated to
the origin. Then the \emph{Euclidean normal space} of $X$ at $x$ is
defined as $N_xX=\{z \in \RR^n:(z-x) \perp (T_xX)_0\}$.

\begin{definition}
\label{def:bottleneck}
A \emph{bottleneck} of a smooth algebraic variety $X \subset \RR^n$ is a pair of
distinct points $(x,y) \in X \times X$ such that $\overline{xy} \subseteq N_xX \cap N_yX$, where $\overline{xy}$ is the line spanned by $x$ and $y$. 
\end{definition}

We note that bottlenecks are given not only by the narrowest parts of the variety, but also by maximally wide parts of the variety, as our algebraic definition considers all critical points rather than just the minimums. 
The bottlenecks of the butterfly curve are shown in Figure \ref{fig:tooth_bottle}.

\begin{definition}
\label{def:bottleneck_distance}
The \emph{narrowest bottleneck distance} $\rho$ of a variety $X \subset \RR^n$ is 
\[ \rho(X)=\inf_{(x,y) \textbf{ a bottleneck}} d(x,y)  \]
where $d(x,y)$ is the Euclidean distance between $x$ and $y$. 
\end{definition}


We will now describe the \emph{bottleneck locus} in
$\RR^{2n}$ which consists of the bottlenecks of $X$ \cite{DEW}.  Let $(f_1,\dots,f_k)
\subseteq \RR[x_1,\dots,x_n]$ be the ideal of $X$. Consider the ring
isomorphism $\phi: \RR[x_1,\dots,x_n] \to \RR[y_1,\dots,y_n]$ defined by $x_i \mapsto
y_i$ and let $f'_i=\phi(f_i)$. Then $f_i$ and $f_i'$ have gradients
$\nabla f_i$ and $\nabla f'_i$ with respect to $\{x_1,\dots,x_n\}$ and
$\{y_1,\dots,y_n\}$, respectively. The \emph{augmented
  Jacobian} $J$ is the following matrix of size $(k+1) \times n$ with
entries in $R=\RR[x_1,\ldots,x_n, y_1,\ldots, y_n]$:
\begin{equation*} \label{eq:augjac}
J = 
\begin{bmatrix}
    y-x \\
    \nabla f_1 \\
    \vdots \\
    \nabla f_k
\end{bmatrix} ,
\end{equation*}
where $y-x$ is the row vector $(x_1-y_1,\dots,x_n-y_n)$. Let $N$
denote the ideal in $R$ generated by $(f_1,\dots,f_k)$ and the
$(n-dim(X)+1) \times (n-dim(X)+1)$ minors of $J$. Then the points $(x,y)$ of the
variety defined by $N$ are the points $(x,y) \in X \times X \subset \RR^{2n}$ such
that $y \in N_xX$. In the same way we define a matrix $J'$ and an
ideal $N' \subseteq R$ by replacing $f_i$ with $f'_i$ and $\nabla f_i$
with $\nabla f'_i$.

The \emph{bottleneck locus} $B$ is the variety
\begin{equation}
 B = V((N + N'): \langle x-y\rangle ^{\infty}) \subset X \times X \subset \RR^{2n}.
\end{equation}
The saturation removes the diagonal, as $(x,y)$ is not a bottleneck if $x=y$. 

\begin{figure}[h]
    \centering
    \includegraphics[height = 3.3 in]{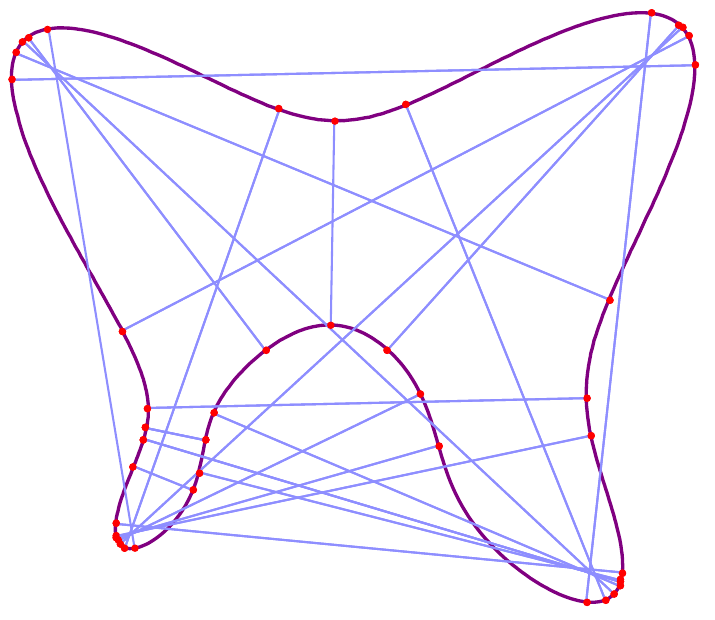}
    \caption{The real bottleneck pairs of the butterfly curve, computed in Example \ref{ex:bottlenecks}.}
    \label{fig:tooth_bottle}
\end{figure}


Next, we give the bottleneck degree, which is 
a measure of the complexity of computing all bottlenecks of an algebraic variety. We refer readers to \cite{Eklund} for a discussion of the numerical algebraic geometry of bottlenecks. Under suitable genericity assumptions described in \cite{DEW}, it coincides with twice the number of bottlenecks of the complexification of $X$. The factor of $2$ is attributed to the fact that in the product $X \times X$, the points $(x,y)$ and $(y,x)$ are distinct, though they correspond to the same pair of points in $X$. 

\begin{theorem}[\cite{DEW}] 
\label{thm:bottleneck_degree}
Under certain genericity assumptions, the degree of the bottleneck locus of a smooth algebraic curve $X \subset \RR^2$ of degree $d$ is
\[ d^4-5d^2+4d. \] 
\end{theorem}

A degree formula for the bottleneck locus of varieties of any dimension is provided in \cite{DEW}. The proof applies the double point formula from intersection theory to a map taking the variety to a variety of its normals. 

\begin{example}
\label{ex:bottlenecks}
We now compute the bottlenecks for the quartic butterfly curve $b(x,y) = 0$. Theorem~\ref{thm:bottleneck_degree} predicts that there are $192/2 = 96$ bottlenecks. Using the description above and \texttt{JuliaHomotopyContinuation} \cite{julia}, we obtain the 96 bottleneck pairs. Of these, 22 are real. We show them in Figure \ref{fig:tooth_bottle}.

\end{example}


We now study bottlenecks from the perspective of Voronoi cells. For a smooth point $p$ in a algebraic curve $X \subset \mathbb{R}^2$, the Voronoi cell $Vor_X(p)$ is a 1-dimensional subset of the normal line to $X$ at the point $p$. Therefore, the normal direction can be recovered from the Voronoi cell $Vor_X(p)$.
For sufficiently small $\epsilon$, an $\epsilon$-approximation $A_\epsilon$ of $X$ will have Voronoi cells whose long edges approximate the normal direction. More precisely, by Theorem \ref{thm:voronoi_convergence}, if $a_\epsilon \in A_\epsilon$ is the point such that $p \in Vor_{A_\epsilon}(a_\epsilon)$, then the directions of the long edges of $Vor_{A_\epsilon}(a_\epsilon)$ converge to the normal direction at $p$.
We remark here that the problem of estimating normal directions from Voronoi cells is well-studied, and numerous efficient, robust algorithms exist \cite{alliez,amenta, feature_detection}.

As in Definition \ref{def:bottleneck}, two points $x,y \in X$ form a bottleneck if their normal lines coincide. This implies that the line connecting them contains both $Vor_X(x)$ and $Vor_X(y)$. 
\begin{definition}

\label{def:bottleneck_candidate}
Let $A_\epsilon$ be an $\epsilon$-approximation of an algebraic curve $X \subset \mathbb{R}^2$. We say a pair $x_\epsilon, y_\epsilon \in A_\epsilon$ is an \emph{approximate bottleneck reach candidate} if the line $\overline{x_\epsilon y_\epsilon}$ joining $x_\epsilon$ and $y_\epsilon$ meets each of $Vor_{A_\epsilon}(x_\epsilon)$ and $Vor_{A_\epsilon}(y_\epsilon)$ at short edges of those cells.
\end{definition}

\begin{figure}[h]
    \centering
    \includegraphics[height = 3.5 in]{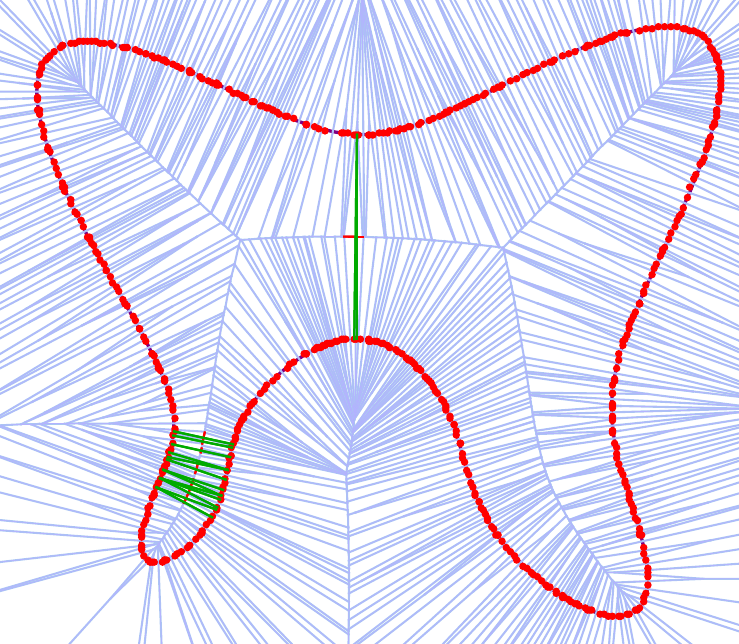}
    \caption{The approximate bottleneck reach candidates (see Definition \ref{def:bottleneck_candidate}) of 568 points sampled from the butterfly curve. The narrowest width of an approximate bottleneck reach candidate is approximately $0.495$ while the true narrowest bottleneck width is approximately $0.503$.}
    \label{fig:approx_bottlenecks}
\end{figure}

In Figure \ref{fig:approx_bottlenecks} we show the approximate bottleneck reach candidates for 348 points sampled from the butterfly curve. The following result gives conditions for bottleneck pairs to be the limit of approximate bottleneck reach candidates.


\begin{theorem}\label{thm:approxbottleneckreachcand}
Let $\{A_\epsilon\}_{\epsilon \searrow 0}$ be a sequence of $\epsilon$-approximations of a smooth algebraic curve $X \subset \mathbb{R}^2$. If $x,y$ is a bottleneck pair of $X$ such that the midpoint of the line segment $\overline{xy}$ is on the medial axis of $X$ and this is the only place where  $\overline{xy}$ meets the medial axis, then there are sequences $x_\epsilon, y_\epsilon \in A_\epsilon$ of approximate bottleneck reach candidates converging to $x$ and $y$. In particular, if $x,y$ is a bottleneck pair of $X$ that achieves the reach, then there are sequences $x_\epsilon, y_\epsilon \in A_\epsilon$ of approximate bottleneck reach candidates converging to $x$ and $y$.
\end{theorem}


\begin{proof}
Let $x,y$ be a bottleneck pair and suppose that the midpoint of the line segment $\overline{xy}$ joining $x$ and $y$ is on the medial axis of $X$ and this is the only place where  $\overline{xy}$ meets the medial axis. We note that this condition holds in particular if $x$ and $y$ are a bottleneck pair that achieves the reach. 
Then $\overline{xy}$ intersects the short edges of exactly two Voronoi cells $Vor_{A_\epsilon}(x_\epsilon)$ and $Vor_{A_\epsilon}(y_\epsilon)$ in a point $v_\epsilon$. Then, $x_\epsilon$ and $y_\epsilon$ form an approximate bottleneck reach candidate by definition. We must then show that the sequence $x_\epsilon$ converges to $x$ and the sequence $y_\epsilon$ converges to $y$.

Since $v_\epsilon$ is in the normal space of $x$,
there exists a neighborhood of $x$ such that
the nearest point of the intersection of this neighborhood and $X$ to $v_\epsilon$ is $x$. So for $\epsilon$ smaller than the radius of this neighborhood, one of the two points in $A_\epsilon$ on either side of $x$ as one moves along $X$ must be the one whose Voronoi cell contains $v_\epsilon$. Since $x_\epsilon$ is the point whose Voronoi cell contains $v_\epsilon$, we have that $x_\epsilon$ is one of the two closest points in $A_\epsilon$ to $x$, meaning that $d(x,x_\epsilon) \leq \epsilon$. Hence, $x_\epsilon$ converges to $x$. Similarly, $y_\epsilon$ converges to $y$.
\end{proof}


\section{Reach}\label{sec:reach}






\begin{example}
\label{ex:reach}
We may find the reach of the butterfly curve by taking the minimum of half the narrowest bottleneck distance and the minimum radius of curvature. This is shown in Figure \ref{fig:reach}. From the computations in Example \ref{ex:bottlenecks}, we find that the narrowest bottleneck distance is approximately $0.251$. Meanwhile, from Example \ref{ex:curvature}, we find that the minimum radius of curvature is approximately $0.104$. Therefore, the reach of the butterfly is approximately $0.104$.
\end{example}

In previous sections, we describe how the reach is the minimum of the minimal radius of curvature and half of the narrowest bottleneck distance. We also give equations for the ideal of the bottlenecks and for the ideal of the critical points of curvature. We now give \texttt{Macaulay2} \cite{M2} code to compute these ideals for smooth algebraic curves $X \subset \RR^2$.
Here, the expression for \texttt{crit} comes from using Lagrange multipliers to find critical points of the affine radius of curvature subject to the constraint given by the curve $f$. 
Finding the points in these ideals, using for example \texttt{JuliaHomotopyContinuation} \cite{julia}, and taking appropriate minimums gives the reach of $X$. See \cite{jreach} for an alternate technique for computing the reach in Julia. 

\begin{verbatim}
R=QQ[x_1,x_2,y_1,y_2]
f= x_1^4 - x_1^2*x_2^2 + x_2^4 - 4*x_1^2 - 2*x_2^2 - x_1 - 4*x_2 + 1
g=sub(f,{x_1=>y_1,x_2=>y_2})
augjacf=det(matrix{{x_1-y_1,x_2-y_2},{diff(x_1,f),diff(x_2,f)}})
augjacg=det(matrix{{y_1-x_1,y_2-x_2},{diff(y_1,g),diff(y_2,g)}})
bottlenecks=saturate(ideal(f,g,augjacf,augjacg),ideal(x_1-y_1,x_2-y_2))
\end{verbatim}

\begin{verbatim}
R=QQ[x,y]
f=x^4 - x^2*y^2 + y^4 - 4*x^2 - 2*y^2 - x - 4*y + 1
num=(diff(x,f))^2 + (diff(y,f))^2
denom=-(diff(y,f))^2*diff(x,diff(x,f)) +
2*diff(x,f)*diff(y,f)*diff(y,diff(x,f)) -
(diff(x,f))^2*diff(y,diff(y,f))
crit=det(matrix({{num*diff(x,denom)- 3/2*denom*diff(x,num),
num*diff(y,denom)-3/2*denom*diff(y,num)},{diff(x,f),diff(y,f)}}))
criticalcurvature=ideal(f,crit)
\end{verbatim}

Alternatively, one can estimate the reach from a point sample. The paper \cite{ACKMRW17} provides a method to do so. We provide a substantially different method that relies upon computing Voronoi and Delaunay cells of points sampled from the curve. We have already discussed how to approximate bottlenecks and curvature using Voronoi cells. This gives the following Voronoi-based Algorithm~\ref{algo:voronoi_reach} for approximating the reach of a curve.

\begin{algorithm}
\caption{Voronoi-Based Reach Estimation}
\label{algo:voronoi_reach}

\begin{algorithmic} 
    \REQUIRE $A \subset X$ a finite set of points forming an $\epsilon-$approximation for a compact, smooth algebraic curve $X \subset \mathbb{R}^2$.
    \ENSURE $\tau$, an approximation of the reach.
    \FOR{$a \in A$}
        \STATE Compute an estimate for the radius of curvature $q_a$ at $a$ using a technique from \cite{feature_detection}.
    \ENDFOR
    \STATE Set $q_{min} = \min_A(q_a)$.
    \STATE Set $\rho$ to be the radius of any disk containing $X$.
    \FOR{$a,b \in A$}
        \IF{$a,b$ form an approximate bottleneck reach candidate as in Definition \ref{def:bottleneck_candidate}}
        \STATE Set $\rho= \min(\rho,d(a,b)/2)$
        \ENDIF
    \ENDFOR
 \STATE Set $\tau = \min(\rho,q_{min})$.
\end{algorithmic}
\end{algorithm}

The reach is equivalently defined as the minimum distance to the medial axis, which suggests the following Delaunay-based Algorithm~\ref{algo:delaunay_reach} for estimating the reach. This algorithm is susceptible to sample error, and to give accurate results would require more sophisticated techniques.

\begin{algorithm}
\caption{Delaunay-Based Reach Estimation}
\label{algo:delaunay_reach}
\begin{algorithmic} 
    \REQUIRE $A \subset X$ a finite set of points forming an $\epsilon-$approximation for a compact, smooth algebraic curve $X \subset \mathbb{R}^2$.
    \ENSURE $\tau$, an approximation of the reach.
    \STATE Compute a Delaunay triangulation $D$ of $A$.
    \STATE Set $M = \emptyset$.
    \FOR{$T$ a Delaunay triangle of $D$}
        \STATE Set $c_T$ be the circumcenter of the Delaunay triangle $T$.
        \STATE Set $M = M \cup \{c_T\}$.
    \ENDFOR
 \STATE Set $\tau = \min_{c \in M, a\in A} d(a,c)$.
\end{algorithmic}
\end{algorithm}

The approximate methods can be used with curves of higher degree, while the symbolic methods are hard to compute for curves with degrees even as low as $4$, but give a more accurate estimate for the reach. This suggests that more work can be done to develop fast and accurate methods to compute the reach of a variety. 


\section*{Acknowledgements}

We thank Paul Breiding, Diego Cifuentes, Yuhan Jiang, Daniel Plaumann, Kristian Ranestad, Rainer Sinn, Bernd Sturmfels, and Sascha Timme for helpful discussions. We thank the referees for their thoughtful comments and suggestions for improving the paper. Research on this project was carried out while the authors were based at the Max Planck Institute for Mathematics in the Sciences (MPI-MiS) in Leipzig, Germany. This material is based upon work supported by the National Science Foundation Graduate Research Fellowship Program under Grant No. DGE 1752814. Any opinions, findings, and
conclusions or recommendations expressed in this material are those of the authors and do not
necessarily reflect the views of the National Science Foundation.

\bibliographystyle{plain}
\bibliography{sample}

\end{document}